\documentclass[12pt,letterpaper]{amsart}

\allowdisplaybreaks
\usepackage{mathpazo}
\usepackage{amsmath,amsfonts,amssymb}
\usepackage{amsrefs}
\usepackage{amsthm}
\usepackage{latexsym,amsmath,amssymb,amsfonts}
\usepackage{rotating}
\usepackage{mathrsfs}
\usepackage{xypic} \xyoption{all}
\usepackage{amscd}
\usepackage{hyperref} 
\usepackage{euscript}
\usepackage{hhline}
\usepackage{graphicx,epstopdf}
\usepackage{epsfig}
\usepackage{xcolor}
\usepackage{textcomp}
\usepackage[all,color]{xy}

\newlength{\fighskip} \fighskip=2pt
\newlength{\figvskip} \figvskip=3pt

\usepackage{hyperref}

\advance\textheight by \topskip
\numberwithin{equation}{section}

\newcommand{\C}{\mathbb{C}}
\newcommand{\N}{\mathbb{N}}

\newcommand{\Z}{\mathbb{Z}}

\newcommand{\W}{\mathcal{W}}

\DeclareMathOperator{\Sym}{\text{Sym}}

\newcommand{\A}{\mathcal A}

\theoremstyle{plain}
\newtheorem{thm}{Theorem}[section]
\newtheorem{thm-defn}{Theorem/Definition}[section]
\newtheorem{lem}[thm]{Lemma}
\newtheorem{lem-defn}[thm]{Lemma/Definition}
\newtheorem{prop}[thm]{Proposition}
\newtheorem{cor}[thm]{Corollary}

\theoremstyle{definition}
\newtheorem{defn}[thm]{Definition}
\newtheorem{notn}[thm]{Notation}

\theoremstyle{remark}
\newtheorem{rmk}[thm]{Remark}

\allowdisplaybreaks[4]  

\begin{document}
\title{Bargmann-Fock sheaves on K\"ahler manifolds}

\author{Kwokwai Chan}
\address{Department of Mathematics, The Chinese University of Hong Kong, Shatin, Hong Kong}
\email{kwchan@math.cuhk.edu.hk}

\author{Naichung Conan Leung}
\address{The Institute of Mathematical Sciences and Department of Mathematics, The Chinese University of Hong Kong, Shatin, Hong Kong}
\email{leung@math.cuhk.edu.hk}

\author{Qin Li}
\address{Shenzhen Institute for Quantum Science and Engineering, Southern University of Science and Technology, Shenzhen, China}
\email{liqin@sustech.edu.cn}

\subjclass[2010]{53D55 (58J20, 81T15, 81Q30)}
\keywords{Deformation quantization, geometric quantization}


\begin{abstract}
	Fedosov used flat sections of the Weyl bundle on a symplectic manifold to construct a star product $\star$ which gives rise to a deformation quantization. By extending Fedosov's method, we give an explicit, analytic construction of a sheaf of Bargmann-Fock modules over the Weyl bundle of a K\"ahler manifold $X$ equipped with a compatible Fedosov abelian connection, and show that the sheaf of flat sections forms a module sheaf over the sheaf of deformation quantization algebras defined $(C^\infty_X[[\hbar]], \star)$. This sheaf can be viewed as the $\hbar$-expansion of $L^{\otimes k}$ as $k \to \infty$, where $L$ is a prequantum line bundle on $X$ and $\hbar = 1/k$.
	
\end{abstract}

\maketitle

\section{Introduction}

This paper is the last in a series of papers \cite{Chan-Leung-Li, CLL}, in which we study the relation between deformation quantization and geometric quantization on a K\"ahler manifold $X$. More precisely, we aim at understanding how deformation quantization acts on geometric quantization via the study of Hilbert space representations of deformation quantization algebras constructed from the space of holomorphic sections $H^0(X, L^{\otimes k})$ of the tensor powers of a prequantum line bundle $L$ on $X$.

For K\"ahler manifolds, deformation quantizations of Wick type is particularly important because they are compatible with the complex structure. 
The most well-known one is the Berezin-Toeplitz quantization \cite{Bordemann-Meinrenken, Ma-Ma}, where one considers a compact K\"ahler manifold $X$ equipped with a prequantum line bundle $L$, as in the setting of geometric quantization. Asymptotic behavior of Toeplitz operators yields the Berezin-Toeplitz star product $\star_{BT}$ on $C^\infty(X)[[\hbar]]$. In \cite{Chan-Leung-Li}, we applied the technique of oscillatory integrals to construct a family $\{H_{x_0}\}$ of formal Hilbert space\footnote{The inner product on $H_{x_0}$ is formal because it takes values in $\C[[\hbar]]$.} representations of the Berezin-Toeplitz quantized algebra $(C^\infty(X)[[\hbar]], \star_{BT})$, parametrized by points $x_0 \in X$.
It is natural to ask, as $x_0$ varies, how the Hilbert spaces $H_{x_0}$ are related. One aim of this paper is to answer this question.


On the other hand, star products on general symplectic manifolds can be obtained by Fedosov's famous construction \cite{Fed, Fedbook}. There have also been extensive studies on Fedosov's construction on K\"ahler manifolds \cite{Bordemann, DLS, Karabegov00,Neumaier}. In \cite{CLL}, we constructed a special family of Fedosov abelian connections on a K\"ahler manifold $X$ as a natural quantization of Kapranov's $L_\infty$ structure \cite{Kapranov}. This gives rise to a star product $\star_\alpha$ for any formal closed $(1,1)$-form $\alpha$ on $X$.
Since $\star_\alpha$ satisfies locality, it defines a sheaf $(C^\infty_X[[\hbar]], \star_\alpha)$ of algebras on $X$, which should be viewed as the ``structure sheaf'' of its {\em quantum geometry}. The identification of the sheaf $C^\infty_X[[\hbar]]$ with the sheaf of flat sections of the Weyl bundle $\W_{X,\C}$ is denoted as $f\longleftrightarrow O_{f}$. 

When the Karabegov form $\omega_\hbar = 2\sqrt{-1} \cdot \omega - \alpha$ of $\star_\alpha$ is real analytic, we consider the subsheaf $(C^{\omega_\hbar}_X[[\hbar]],\star_\alpha)$ of smooth functions satisfying a real analytic condition (Definition \ref{definition: convergence-property-analytic-function}).
In Section \ref{section: module-sheaf}, we explicitly construct a sheaf of Bargmann-Fock modules over $\W_{X,\C}$ which is equipped with a compatible Fedosov abelian connection. We then prove that the subsheaf $\mathcal{F}_{X,\alpha}^{\text{flat}}$ of convergent flat sections, which we call the {\em Bargmann-Fock sheaf} $\mathcal{F}_{X,\alpha}^{\text{flat}}$ (see Definition \ref{definition: flat-Bargmann-Fock}), forms a module over the sheaf of analytic functions:
\begin{thm}[=Theorem \ref{theorem: Bargmann-Fock-representation}]\label{theorem: BF-rep-intro}
	The Bargmann-Fock sheaf $\mathcal{F}_{X,\alpha}^{\text{flat}}$ is a sheaf of modules over the sheaf $\left(C^{\omega_\hbar}_X[[\hbar]],\star_\alpha\right)$. 
\end{thm}

We denote the action of $f\in (C^{\omega_\hbar}_X[[\hbar]],\star_\alpha)$ on $\mathcal{F}_{X,\alpha }^{\text{flat}}$ by $O_f \circledast -$.
Our next theorem shows that this action is given by {\em formal Toeplitz operators} introduced in \cite[Definition 2.24]{Chan-Leung-Li}. These are defined as compositions of multiplication and orthogonal projection operators on formal Hilbert spaces, so they are formal analogues of the usual Toeplitz operators.
The formal Hilbert space relevant to us here is a subspace $V_{x_0}$ of the stalk $(\mathcal{F}_{X,\alpha}^{\text{flat}})_{x_0}$ at a point $x_0 \in X$, which is isomorphic to the space of germs of formal holomorphic functions at $x_0$, i.e., $V_{x_0}\cong\mathcal{O}_{X,x_0}[[\hbar]]$ (see Proposition \ref{proposition: representation-stalk}). 
A germ $\Psi _{s}\in V_{x_0}$ can be identified with a germ of holomorphic function $s\in \mathcal{O}_{X,x_0}\left[ \left[ h\right] \right] $ via
$$\Psi _{s}=J_{s}\cdot e^{\beta /h}\otimes e_{x_0},$$
where $J_{s}$ is the jet of $s$ at $x_0$ expressed in K-coordinates. 
More detailed explanation of the notations can be found in Section \ref{section: star-product-formal-Toeplitz}. 
\begin{thm}[=Theorem \ref{theorem: Bargmann-Fock-microlocal-Toeplitz}]
Given  $f\in C^{\omega_\hbar}_X[[\hbar]]$ and $\Psi _{s}\in V_{x_{0}}$ for any $x_{0}\in X$, we have
$$
O_{f}\circledast \Psi _{s}=\Psi _{s^{\prime }},
$$
where $s^{\prime }$ is obtained by applying the formal Toeplitz operator associated to $f$ on $s$, namely,
$$
T_{(J_{f})_{x_0},\Phi }(J_s) = J_{s^{\prime }}.
$$
\end{thm}

Our construction and proof of Theorem \ref{theorem: BF-rep-intro}, which are analytic in nature, follow Fedosov's original approach closely. Note that the module sheaves in Theorem \ref{theorem: BF-rep-intro} exist even when $X$ is not pre-quantizable. On the other hand, closely related studies on such module sheaves have been carried out using deformation-obstruction theory. In the real symplectic manifolds context, such constructions were established in the work of Nest-Tsygan \cite{Nest-Tsygan} and Tsygan \cite{Tsygan}. In \cite{BGKP}, Baranovsky, Ginzburg, Kaledin and Pecharich gave a deformation theoretic construction of quantizations of line bundles as module sheaves in the algebraic setting.

\subsection*{Acknowledgement}
\

We thank Si Li and Siye Wu for useful discussions, and the anonymous referees for valuable comments. The first named author thanks Martin Schlichenmaier and Siye Wu for inviting him to attend the conference GEOQUANT 2019 held in September 2019 in Taiwan, in which he had stimulating and very helpful discussions with both of them as well as Jørgen Ellegaard Andersen, Motohico Mulase, Georgiy Sharygin and Steve Zelditch.

K. Chan was supported by grants of the Hong Kong Research Grants Council (Project No. CUHK14302617 \& CUHK14303019) and direct grants from CUHK.
N. C. Leung was supported by grants of the Hong Kong Research Grants Council (Project No. CUHK14301117 \& CUHK14303518) and direct grants from CUHK.
Q. Li was supported by Guangdong Basic and Applied Basic Research Foundation (Project No. 2020A1515011220).

\section{Preliminaries on Fedosov deformation quantization}

Throughout this paper, we assume that $X$ is a K\"ahler manifold of complex dimension $n$, with $\omega$ denoting its K\"ahler form. We first recall the definition of Wick type deformation quantization on a K\"ahler manifold:
\begin{defn}\label{definition: Wick-type-deformation-quantization}
A {\em deformation quantization} on $X$ is a $\mathbb{R}[[\hbar]]$-bilinear, associative product $\star$ on  $C^\infty(X)[[\hbar]]$ of the form
$$
f\star g=f\cdot g+\sum_{k\geq 0}\hbar^k C_k(f,g),
$$
where $C_k(\cdot, \cdot)$'s are bi-differential operators on $X$, such that
$$f\star g-g\star f=\hbar\{f,g\}+O(\hbar^2).$$
It is said to be {\em of Wick type} if the bi-differential operators $C_k(\cdot, \cdot)$'s take holomorphic and anti-holomorphic derivatives of the first and second arguments respectively.
\end{defn}

In this section, we give a very brief review of Fedosov's construction of Wick type deformation quantizations on a K\"ahler manifold $X$; we refer to \cite{Fed} for more details.

\begin{defn}\label{definition: Wick-algebra}
The {\em Wick product} on the space $\mathcal{W}_{\mathbb{C}^n}:=\mathbb{C}[[z^1,\bar{z}^1,\cdots,z^n,\bar{z}^n]][[\hbar]]$ is defined by 
 \begin{equation*}\label{equation: defn-Wick-product}
  f\star g:=\exp\left(-\hbar\sum_{i=1}^n\frac{\partial}{\partial z^i}\frac{\partial}{\partial \bar{w}^i}\right)(f(z,\bar{z})g(w,\bar{w}))\Big|_{z=w}
 \end{equation*}
\end{defn}
We assign a $\Z$-grading on $\mathcal{W}_{\mathbb{C}^n}$ by letting the monomial $\hbar^k z^I\bar{z}^J$ to have total degree $2k+|I|+|J|$. On a K\"ahler manifold $X$, we consider the following Weyl bundles:
\begin{align*}\label{equation: Weyl-bundle}
 \W_{X}& := \widehat{\Sym}T^*X[[\hbar]], \quad \overline{\W}_X:=\widehat{\Sym}\overline{T^*X}[[\hbar]],\\
 \W_{X,\C}& := \W_{X}\otimes\overline{\W}_X=\widehat{\Sym}T^*X_{\C}[[\hbar]].
\end{align*}
The fiberwise Hermitian structure on the complexified tangent bundle $TX_{\C}$ enables us to define a fiberwise (non-commutative) Wick product on $\W_{X,\C}$. Under a local holomorphic coordinate system $z=(z^1,\cdots, z^n)$, a section of $\W_{X,\C}$ is written as
$$
\alpha=\sum_{I,J}\alpha_{I\bar{J}}y^I\bar{y}^J,
$$
where the sum is over all multi-indices. Writting $\omega=\omega_{i\bar{j}}dz^i\wedge d\bar{z}^j$, then we have
\begin{equation*}\label{equation: fiberwise-Wick-product}
 \alpha\star\beta:=\sum_{k\geq 0}\frac{1}{k!}\cdot\left(\frac{\sqrt{-1}\cdot\hbar}{2}\right)^k\omega^{i_1\bar{j}_1}\cdots\omega^{i_k\bar{j}_k}\frac{\partial^k\alpha}{\partial y^{i_1}\cdots\partial y^{i_k}}\frac{\partial^k\beta}{\partial \bar{y}^{j_1}\cdots\partial \bar{y}^{j_k}}.
\end{equation*}
There is the natural {\em symbol map} which takes the constant term of a formal power series:
$$
\sigma: \Gamma(X,\W_{X,\C})\rightarrow C^\infty(X)[[\hbar]].
$$
We also introduce several operators on $\A_X^\bullet(\mathcal{W}_{X,\mathbb{C}})$:
\begin{defn}
We define the following natural operators acting as derivations on $\A_X^\bullet(\mathcal{W}_{X,\mathbb{C}})$:
\begin{align*}
\delta^{1,0} a := dz^i\wedge\frac{\partial a}{\partial y^i},\quad 
\delta^{0,1}a := d\bar{z}^j\wedge\frac{\partial a}{\partial\bar{y}^j},\quad \delta:=\delta^{1,0}+\delta^{0,1},
\end{align*}
as well as
\begin{align*}
(\delta^{1,0})^*a  := y^i\cdot \iota_{\partial_{z^i}}a, \quad
(\delta^{0,1})^*a  := \bar{y}^j\cdot \iota_{\partial_{\bar{z}^j}}a, \quad \delta^*:=(\delta^{1,0})^*+(\delta^{0,1})^*
\end{align*}
We also define the operators $(\delta^{1,0})^{-1}$ and $(\delta^{0,1})^{-1}$ by normalizing $(\delta^{1,0})^{*}$ and $(\delta^{1,0})^{*}$ respectively: 
\begin{align*}
 (\delta^{1,0})^{-1}:=&\frac{1}{p_1+p_2}(\delta^{1,0})^*\ \text{on $\A_X^{p_1,q_1}(\W_{X,\C})_{p_2,q_2}$},\\(\delta^{0,1})^{-1}:=&\frac{1}{q_1+q_2}(\delta^{0,1})^*\ \text{on $\A_X^{p_1,q_1}(\W_{X,\C})_{p_2,q_2}$},\\
\delta^{-1}:=&\frac{1}{p+q}\delta^*\ \text{on $\A_X^{p}(\W_{X,\C})_q$}.
\end{align*}
\end{defn}

Following \cite{Fed}, we define the following extension of the Wick algebra:\footnote{Note that the extension of the Weyl algebra considered in \cite{Chan-Leung-Li} is {\em different} from the one here.}
\begin{defn}[p.224 in \cite{Fed}]\label{defn:extension-Wick-algebra}
The extention $\mathcal{W}_{\mathbb{C}^n}^+$ of $\mathcal{W}_{\mathbb{C}^n}$ is defined as follows:
\begin{itemize}
 \item Elements $U\in\mathcal{W}_{\mathbb{C}^n}^+$ are given by power series, possibly with negative powers of $\hbar$.
 \item For any element $U \in \mathcal{W}_{\mathbb{C}^n}^+$, the total degree $2k+|I|+|J|$ of every term is nonnegative.
 \item For any element $U \in \mathcal{W}_{\mathbb{C}^n}^+$, there are only a finite number of terms with a given nonnegative total degree.
\end{itemize}
\end{defn}
The extension $\mathcal{W}_{\mathbb{C}^n}^+$ is closed under the Wick product, and we can define the corresponding extended Weyl bundle $\W_{X,\C}^+$ on $X$. 

A {\em Fedosov abelian connection} on the Weyl bundle $\W_{X,\C}$ is a connection of the form
\begin{equation*}\label{equation: Fedosov-connection-Weyl-bundle}
 D_F=\nabla-\delta+\frac{1}{\hbar}[I,-]_\star
\end{equation*}
which is flat, i.e., $D_F^2=0$; here $\nabla$ is the Levi-Civita connection on $X$, $[-,-]_\star$ denotes the bracket associated to the Wick product, and $I\in\A^1_X(\W_{X,\C})$. 
Note that $\nabla^2=\frac{1}{\hbar}[R_\nabla,-]_\star$, where $R_\nabla=-2\sqrt{-1}R_{i\bar{j}k}^m\omega_{m\bar{l}}dz^i\wedge d\bar{z}^j\otimes y^k\bar{y}^l$.
From this we see that the flatness condition $D_F^2=0$ is equivalent to the {\em Fedosov equation}:
\begin{equation}\label{equation: Fedosov-equation-form-1}
\nabla I-\delta I+\frac{1}{\hbar}I\star I+R_\nabla=-\alpha\in\hbar\cdot \A_X^2(X)[[\hbar]].
\end{equation}
Let $\Gamma^{\text{flat}}(X,\W_{X,\C})$ be the space of flat sections of the Weyl bundle under $D_F$. It is shown in \cite{Fed} that the symbol map $\sigma: \W_{X,\C}\rightarrow C^\infty(X)[[\hbar]]$ induces the following isomorphism:
$$
\Gamma^{\text{flat}}(X,\W_{X,\C})\overset{\sim}{\rightarrow}C^\infty(X)[[\hbar]].
$$
If we denote by $O_f$ the flat section of $\W_{X,\C}$ corresponding to a formal smooth function $f$, then the associated star product can be defined by
$$O_{f\star g}:=O_f\star O_g.$$

\subsection{$L_\infty$ structure on K\"ahler manifolds: classical and quantum}
\

\noindent In this subsection, we first recall Kapranov's $L_\infty$-algebra structure on a K\"ahler manifold and its geometric interpretation. Then we describe its classical and quantum extensions. As discovered in \cite{CLL}, the latter gives rise to a special class of Fedosov connections. 

 Explicitly, the $L_\infty$ structure is equivalent to the following flat connection on $\W_{X}$:
\begin{equation*}\label{equation: L-infty-structure}
 D_K=\nabla-\delta^{1,0}+\sum_{n\geq 2}\tilde{R}_n^*, 
\end{equation*}
where the subscript ``K'' stands for ``Kapranov''.
Here $\tilde{R}_n^*$'s are defined by extending the following $R_n^*$'s to $\A_X^\bullet$-linear derivations on $\W_X$:
$$
R_2^*=\frac{1}{2}R_{i\bar{j}k}^m d\bar{z}^j\otimes (y^iy^k\otimes\partial_{y^m}),\quad R_n^*=(\delta^{1,0})^{-1}\circ{\nabla}^{1,0}(R_{n-1}^*)\text{ for $n>2$},
$$
where $R_{i\bar{j}k}^m$'s are the coefficients of the curvature tensor. We write these $R_n^*$'s locally as
$$
R_n^*=R_{i_1\cdots i_n,\bar{l}}^jd\bar{z}^l\otimes (y^{i_1}\cdots y^{i_n}\otimes\partial_{y^j}).
$$
\begin{rmk}
These $\A_X^\bullet$-linear operators $\tilde{R}_n^*$'s extend naturally to the complexified Weyl bundle $\A_X^\bullet(\W_{X,\C})$.
\end{rmk}

\begin{notn}To simplify notations in later computations, we introduce two operators:
$$
\tilde{\nabla}^{1,0}:=(\delta^{1,0})^{-1}\circ\nabla^{1,0},\quad
\tilde{\nabla}^{0,1}:=(\delta^{0,1})^{-1}\circ\nabla^{0,1}.
$$
\end{notn}
The symbol map of the Weyl bundle gives rise to an isomorphism:
\begin{prop}\label{proposition:flat-sections-holomorphic-Weyl}
For every $f\in\mathcal{O}_X(U)$, there exists a unique flat section $J_f$ (where ``J'' stands for ``jets'') under the connection $D_K$ such that $\sigma(J_f)=f$. Explicitly:
\begin{equation*}\label{equation: jet-holomorphic-function}
J_f=\sum_{k\geq 0}(\tilde{\nabla}^{1,0})^k(f). 
\end{equation*}
Thus the space of (local) flat sections of the holomorphic Weyl bundle with respect to the connection $D_K$ is isomorphic to the space of holomorphic functions. 
\end{prop}
It was shown by Bochner that there exists the following special holomorphic coordinate system at each point $x_0\in X$ when the K\"ahler form $\omega$ is real analytic. 
\begin{defn}\label{definition: K-coordinates-K-frame}
	 A holomorphic coordinate system $(z_1,\cdots, z_n)$ centered at $x_0\in X$ is called a {\em K\"ahler normal coordinate} if there exists a unique function $\rho_{x_0}$ around $x_0$ such that $\partial\bar{\partial}(\rho_{x_0})=-2\sqrt{-1}\omega$ and whose Taylor expansion at $x_0$ is of the form
 	\begin{equation}\label{equation: Taylor-expansion-hermitian-metric}
 		\rho_{x_0}(z,\bar{z})\sim -2\sqrt{-1}\cdot \omega_{i\bar{j}}(x_0)z^i\bar{z}^j+\sum_{|I|,|J|\geq 2}\frac{1}{|I|!|J|!}\frac{\partial^{|I|+|J|}\rho_{x_0}}{\partial z^I\bar{z}^J}(x_0)z^I\bar{z}^J.
 	\end{equation}
	 If $-2\sqrt{-1}\omega_{i\bar{j}}(x_0)=\delta_{ij}$, then we call this a {\em K-coordinate centered at $x_0$}.
\end{defn}
The geometric meaning of the connection $D_K$ is that the germ $(J_f)_{x_0}$ is precisely the Taylor expansion of $f$ at $x_0$ under the K\"ahler normal coordinates. 

We now introduce two extensions of $D_K$ to $\W_{X,\C}$: one quantum and the other classical. For the quantum extension, we first use the K\"ahler form to ``lift the last subscript'' of $R_n^*$ and define
$$
I_n:=-2\sqrt{-1}\cdot R_{i_1\cdots i_n,\bar{l}}^j\omega_{j\bar{k}}d\bar{z}^l\otimes (y^{i_1}\cdots y^{i_n}\bar{y}^{k})\in\mathcal{A}_X^{0,1}(\mathcal{W}_{X,\mathbb{C}}). 
$$
Then we let $I:=\sum_{n\geq 2}I_n$. In \cite{CLL}, we proved the following theorem:
\begin{thm}\label{theorem: Fedosov-connection-general}
Suppose $\alpha$ is a representative of a formal cohomology class in $\hbar H^2_{dR}(X)[[\hbar]]$ of type $(1,1)$. Let $\varphi$ be a (locally defined) function such that $\partial\bar{\partial}\varphi=\alpha$ and set $J_\alpha:=\sum_{k\geq 1}(\tilde{\nabla}^{1,0})^k(\bar{\partial}\varphi)$. Then we have
\begin{enumerate}
  \item $I_\alpha := I+ J_\alpha \in \mathcal{A}_X^{0,1}(\mathcal{W}_{X,\mathbb{C}})$ is a solution of the Fedosov equation, i.e.,
 	\begin{equation*}
	\nabla I_\alpha - \delta I_\alpha + \frac{1}{\hbar} I_\alpha\star I_\alpha + R_\nabla=-\alpha.
	\end{equation*}
  We denote the corresponding Fedosov abelian connection by $D_{F,\alpha}$ and the the corresponding Fedosov star product by $\star_\alpha$.
  \item The Fedosov connection $D_{F,\alpha}$ is an extension of $D_K$, i.e.,  $D_{F,\alpha}|_{\W_X}=D_K$. 
  \item Every star product on $X$ of Wick type can be obtained from such Fedosov connections. 
\end{enumerate}
\end{thm}
Let $\gamma_\alpha=I_\alpha+2\sqrt{-1}\omega_{i\bar{j}}(dz^i\otimes\bar{y}^j-d\bar{z}^j\otimes y^i)$. Then we can also write the Fedosov abelian connection $D_{F,\alpha}$ as
\begin{equation}\label{equation: Fedosov-connection-gamma-alpha}
 D_{F,\alpha}=\nabla+\frac{1}{\hbar}[\gamma_\alpha,-]_\star, 
\end{equation}
If we put $\omega_\hbar:=2\sqrt{-1}\omega-\alpha$, then the Fedosov equation \eqref{equation: Fedosov-equation-form-1} is equivalent to 
\begin{equation*}\label{equation: Fedosov-equation-gamma-alpha}
 \nabla \gamma_\alpha  + \frac{1}{\hbar} \gamma_\alpha\star \gamma_\alpha + R_\nabla=\omega_\hbar.
\end{equation*}

On the other hand, the complex conjugate of the connection $D_K$ is a flat connection $\overline{D}_K$ on $\overline{\mathcal{W}}_X$. Then
$$D_C := D_K \otimes 1+1\otimes\overline{D}_K$$
is naturally a flat connection on $\mathcal{W}_{X,\mathbb{C}}=\W_{X}\otimes\overline{\W}_X$ such that $D_C|_{\W_X}=D_K$. This gives the classical extension of $D_K$ (where the subscript $C$ stands for ``classical'').
The motivation behind this extension is very simple: since the flat sections with respect to $D_K$ correspond to (local) holomorphic functions on $X$, by adding the anti-holomorphic components in $\mathcal{W}_{X,\mathbb{C}}$, we shall see all the smooth functions. This is indeed the case.
\begin{prop}\label{proposition: classical-flat-section-smooth-function}
There is a one-to-one correspondence between $C^\infty(X)[[\hbar]]$ and the space of flat sections of the Weyl bundle $\W_{X,\mathbb{C}}$ with respect to the flat connection $D_C$. 
\end{prop}
\begin{proof}
Given any smooth function $f$, we need to show that there exists a unique $J_f$ such that  $\sigma(J_f)=f$ and $D_C(J_f)=0$, where $\sigma$ is the symbol map. The proof is very similar to that of Theorem 3.3 in Fedosov \cite{Fed}, so we will be brief. 
For the uniqueness of $J_f$, consider a nonzero section $s$ of $\W_{X,\mathbb{C}}$ with $\sigma(s)=0$. Let $s_0$ be the terms in $s$ of the smallest weight. Then $\delta(s_0)$ is nonzero and of smaller weight, so $s$ cannot be flat.
For the existence of $J_f$, consider the filtration on $\A_X^\bullet(\mathcal{W}_{X,\mathbb{C}})$ induced by the polynomial degrees of terms in $\W_{X,\C}$. The fact that the fiberwise de Rham differential $\delta$ has cohomology concentrated in degree $0$ implies the existence of $J_f$, which is uniquely determined by the iterative equation
\begin{equation}\label{equation: classical-flat-section-iteration-equation-no-type}
J_f=f+\delta^{-1}(D_C+\delta)(J_f).
\end{equation}
\end{proof}
We now give explicit formula for some terms in $J_f$. The first observation from the iterative equation \eqref{equation: classical-flat-section-iteration-equation-no-type} is that 
\begin{equation*}\label{equation: classical-flat-section-n-0-term}
(J_f)_{n,0}=(\tilde{\nabla}^{1,0})^n(f).
\end{equation*}

\begin{prop}\label{proposition: flat-section-0-1-part-flat-connection}
Let $D_C^{0,1}$ denote the $(0,1)$-part of the flat connection $D_C$. Then there is a one-to-one correspondence between $\ker(D_C^{0,1})$ and smooth sections of the holomorphic Weyl bundle $\mathcal{W}_X$.
\end{prop}
\begin{proof}
Consider the projection map
$$
\pi_{*,0}:\mathcal{W}_{X,\mathbb{C}}\rightarrow\mathcal{W}_X.
$$
and the filtration induced by the degrees of anti-holomorphic components on the Weyl bundle. Then the statement of this proposition is simply that there exists a unique section of $\mathcal{W}_{X,\mathbb{C}}$ annihilated by $D_C^{0,1}$, with a prescribed leading term, the proof of which is again very similar to that  of \cite[Theorem 3.3]{Fed}. 
\end{proof}
Proposition \ref{proposition: flat-section-0-1-part-flat-connection} and the uniqueness of $J_f$ implies the following corollary:
\begin{cor}\label{corollary: iteration-equation-classical-flat-section-function}
$J_f$ is determined uniquely and iteratively by the following conditions:
\begin{itemize}
\item The $(k,0)$-component of $J_f$ is given by $(\tilde{\nabla}^{1,0})^k(f)$.
\item For $n\geq 0$, we have
\begin{equation}\label{equation: classical-flat-section-iteration-equation-with-types}
(J_f)_{*,n+1} = ((\delta^{0,1})^{-1}\circ(D_C^{0,1}+\delta^{0,1}))((J_f)_{*,n}).
\end{equation}
\end{itemize}
\end{cor}

\begin{lem}\label{lemma: Taylor-expansion-k-1-terms}
Let $f$ be a smooth function on $X$, then the $(1,k)$- and $(k,1)$-components of $J_f$ are given respectively by $(J_f)_{1,k}  = (\tilde{\nabla}^{0,1})^k\circ\tilde{\nabla}^{1,0})(f)$ and $(J_f)_{k,1}  = ((\tilde{\nabla}^{1,0})^k\circ\tilde{\nabla}^{0,1})(f)$.
\end{lem}
\begin{proof}
We will only prove the formula for $(J_f)_{1,k}$ by induction because the other can be proven similarly. It is obvious for $k=0$. Suppose that the statement is valid for $k\leq n$. Then from equation \eqref{equation: classical-flat-section-iteration-equation-with-types}, we have
\begin{align*}
(J_f)_{1,k+1} & = ((\delta^{0,1})^{-1}\circ(D_C^{0,1}+\delta^{0,1}))((J_f)_{*,k})\\
& = ((\delta^{0,1})^{-1}\circ(\nabla^{0,1}+\sum_{n\geq 2}\tilde{R}_n^*))((J_f)_{*,k})\\
& = ((\delta^{0,1})^{-1}\circ\nabla^{0,1})((J_f)_{1,k}),
\end{align*}
where the last equality follows from the fact that $\tilde{R}_n^*((J_f)_{*,k})$ has holomorphic degree greater than or equal to $2$.
\end{proof}

\subsection{Sections of $\W_{X,\C}$ associated to closed $(1,1)$-forms}
\

We consider here a section of the Weyl bundle associated to a closed $(1,1)$-form on $X$. We use the symplectic form as an example: Let $\varphi$ be a (locally defined) function  on $X$ such that $\partial\bar{\partial}\varphi=\omega$, which is unique up to the sum of a purely holomorphic and a purely anti-holomorphic function. It follows that the components of $J_{\varphi}$ of mixed type only depend on $\omega$. We denote those mixed terms in $J_\varphi$ by $\Phi_\omega$:
\begin{equation*}\label{equation: Phi}
\Phi_\omega := \sum_{i,j\geq 1}(J_\varphi)_{i,j}.
\end{equation*}

It is clear that $(\Phi_\omega)_{1,1}=\frac{\partial^2\varphi}{\partial z^i\partial\bar{z}^j}y^i\bar{y}^j=\omega_{i\bar{j}}y^i\bar{y}^j$. 
\begin{lem}\label{lemma: Taylor-expansion-Kahler-potential-vanishing-k-1-terms}
We have $(\Phi_\omega)_{1,k}=(\Phi_\omega)_{k,1}=0$ for $k\geq 2$.
\end{lem}
\begin{proof}
By Lemma \ref{lemma: Taylor-expansion-k-1-terms}, we have
$(J_{\varphi})_{k,1}=((\tilde{\nabla}^{1,0})^k\circ \tilde{\nabla}^{0,1})(\varphi)=(\tilde{\nabla}^{1,0})^{k-1}(\omega_{i\bar{j}}y^i\bar{y}^j)=0$,
where the last equality follows from the fact that the K\"ahler form is parallel with respect to the Levi-Civita connection $\nabla$. The vanishing for the $(1,k)$ terms is similar. 
\end{proof}
For later computations, we give a formula for $(\Phi_\omega)_{n,2}, n\geq 2$. Using Corollary \ref{corollary: iteration-equation-classical-flat-section-function} and the fact that $(J_\varphi)_{n,1}=0$ for $n\geq 2$, we have
\begin{align*}
(\Phi_\omega)_{n,2}=(J_\varphi)_{n,2}
& = (\delta^{0,1})^{-1}\left(\nabla^{0,1}((J_\varphi)_{n,1}) + \sum_{i=2}^n\tilde{R}_i^*((J_\varphi)_{n-i+1,1})\right)\\
& = (\delta^{0,1})^{-1} \left(\nabla^{0,1}((J_\varphi)_{n,1})+\tilde{R}_n((J_\varphi)_{1,1}) \right)\\
& = (\delta^{0,1})^{-1}\circ\tilde{R}_n((J_\varphi)_{1,1})\\
& = (\delta^{0,1})^{-1}\left(\tilde{R}_n^*(\omega^{i\bar{j}}y^i\bar{y}^j)\right)\\
& = (\delta^{0,1})^{-1}\left(R^k_{i_1\cdots i_n,\bar{l}}d\bar{z}^l\otimes(y^{i_1}\cdots y^{i_n})\frac{\partial}{\partial y^k}(\omega_{i\bar{j}}y^i\bar{y}^j)\right)\\
& = (\delta^{0,1})^{-1}\left(R^k_{i_1\cdots i_n,\bar{l}}\omega_{k\bar{j}}d\bar{z}^l\otimes(y^{i_1}\cdots y^{i_n}\bar{y}^j)\right)
\end{align*}
From the above computation, we obtain
\begin{equation}\label{equation: relation-Phi-I}
\delta^{0,1}((\Phi_\omega)_{n,2})=R^k_{i_1\cdots i_n,\bar{l}}\omega_{k\bar{j}}d\bar{z}^l\otimes(y^{i_1}\cdots y^{i_n}\bar{y}^j)=\frac{\sqrt{-1}}{2}I_{n}\text{ for $n\geq 2$}. 
\end{equation}
In general, let $\alpha$ be a representative of a formal cohomology class in $\hbar H^2_{dR}(X)[[\hbar]]$ of type $(1,1)$, and let $\varphi$ be a (local) potential of $\alpha$. We set $\Phi_{\alpha} := \sum_{i,j\geq 1}(J_\varphi)_{i,j}$. Then we have
$$
J_\alpha=\delta^{0,1}\left(\sum_{k\geq 1}(\Phi_\alpha)_{k,1}\right).
$$
The following theorem, whose proof will be given in Appendix \ref{appendix: proof-theorem-jets-flat-section}, describes the relation between the classical and quantum flat sections:
\begin{thm}\label{thm:main-thm-formal-Toeplitz-Fedosov-equation}
	Let $\Phi:=2\sqrt{-1}\left(-\omega_{i\bar{j}}y^i\bar{y}^j+\Phi_\omega\right)-\Phi_\alpha$.
	Given a smooth function $f$, let $O_f$ be the flat section under the Fedosov connection $D_{F,\alpha}$ associated to $f$, i.e., $D_{F,\alpha}(O_f)=0$. Then we have
\begin{equation}\label{equation: formal-Toeplitz-operator-equals-flat-section}
J_f\cdot e^{\Phi/\hbar}=e^{\Phi/\hbar}\star O_f.
\end{equation}
\end{thm}
The germ $(J_f)_{x_0}$ is the Taylor expansion of $f$ at $x_0$ under K\"ahler normal coordinates and their complex conjugates, which is a classical object. On the other hand, the flat section $O_f$ is a quantum object which we will explain in Section \ref{section: micri-local-Berezin-Toeplitz}.  In particular, for holomorphic functions $f\in\mathcal{O}(U)$, since $D_K=D_C|_{\W_X}=D_{K,\alpha}|_{\W_X}$, we must have $J_f=O_f$ (which can also be seen from equation \eqref{equation: formal-Toeplitz-operator-equals-flat-section}). This says that holomorphic functions {\em do not receive any quantum corrections}. 




\section{Bargmann-Fock sheaf}\label{section: module-sheaf}

Since a deformation quantization defined via Fedosov abelian connections satisfies {\em locality}, it defines a sheaf of algebras on $X$, which can be viewed as the structure sheaf of the ``quantum geometry'' on $X$. The goal of this section is to construct a sheaf of modules over this structure sheaf, which we call a {\em Bargmann-Fock sheaf}.


\subsection{Extended holomorphic Weyl bundle and formal line bundles}\label{subsection: Bargmann-Fock-representation}
\

In this subsection, we define the extended holomorphic Weyl bundle on a K\"ahler manifold $X$. This is the first step in the construction of a Bargmann-Fock sheaf.  We first recall the Bargmann-Fock representation of the Wick algebra:
\begin{defn}\label{definition: Fock-representation-of-Wick-algebra}
We define an action of a monomial $f = z^{\alpha_1}\cdots z^{\alpha_k}\bar{z}^{\beta_1}\cdots\bar{z}^{\beta_l} \in \mathcal{W}_{\C^n}$ on $s\in\mathcal{F}_{\C^n} := \C[[z^1,\cdots,z^n]][[\hbar]]$ by 
\begin{equation}\label{equation: Toeplitz-operators-C-n}
f\circledast s:=\hbar^l\frac{\partial}{\partial z^{\beta_1}}\circ\cdots\circ\frac{\partial}{\partial z^{\beta_l}}\circ m_{z^{\alpha_1}\cdots z^{\alpha_k}}(s),
\end{equation}
where $m_{z^{\alpha_1}\cdots z^{\alpha_k}}$ denotes the multiplication by $z^{\alpha_1}\cdots z^{\alpha_k}$. It is known that
$$
f\circledast(g\circledast s)=(f\star g)\circledast s,
$$
so this defines an action of the Weyl algebra $\mathcal{W}_{\C^n}$ on $\mathcal{F}_{\C^n}$, known as the {\em Bargmann-Fock representation} (or the {\em Wick normal ordering} in physics literature). 
\end{defn}

Via the fiberwise Bargmann-Fock action, the holomorphic Weyl bundle $\W_X$ can be regarded as a sheaf of $\W_{X,\C}$-modules. We consider the following extension of $\W_X$ by allowing formal exponentials:
\begin{defn}\label{definition: Weyl-bundle-with-exponentials}
We define the sheaf $\W_{X,e}$ of extended  Weyl algebra with exponentials as follows:
for every open set $U\subset X$, we consider the space of finite sum of pairs
$$
\sum_{i=1}^k(f_i, e^{g_i/\hbar}),
$$
where $f_i,g_i$'s are smooth sections of $\W_X$ on $U$. We define the multiplication by the linear extension of 
$$
(f_1,e^{g_1/\hbar})\cdot(f_2,e^{g_2/\hbar}):=(f_1f_2,e^{(g_1+g_2)/\hbar}) 
$$
These are subject to the equivalence relation that $(f_1,e^{g_1/\hbar})\sim (f_2,e^{g_2/\hbar})$ if $f_1=f_2$ and $g_1-g_2\in\C[[\hbar]]$. Then the space $\W_{X,e}(U)$ of sections of $\W_{X,e}$ over $U$ is given by the set of equivalence classes.
\end{defn}

There is a sub-sheaf $\mathcal{O}_{X,e}$ of $\W_{X,e}$ defined as follows: for an open set $U\subset X$, the space $\mathcal{O}_{X,e}(U)$ consists of equivalence classes of finite sums
$$
\sum_{i=1}^n(f_i, e^{g_i/\hbar}),
$$
where $f_i,g_i\in \mathcal{O}_X(U)[[\hbar]]$ are all formal holomorphic functions on $U$. We should point out that $\mathcal{O}_X$  is naturally a sub-sheaf of $\mathcal{O}_{X,e}$ ($\W_{X,e}$) by the inclusion
$$
f\mapsto (f,e^{0/\hbar}).
$$
\begin{notn}
 For convenience, we will use the notation $f\cdot e^{g/\hbar}$ for the the pair $(f,e^{g/\hbar})$ in $\W_{X,e}$ (and also $\mathcal{O}_{X,e}$). 
\end{notn}
Similar to the definition holomorphic line bundles, we can define the notions of formal holomorphic line bundles, which are similar to local line bundles in \cite{Melrose} and twisting bundles in \cite{Tsygan}.  
\begin{defn}
A {\em formal line bundle} on $X$ is an invertible $\mathcal{O}_{X,e}$-module. 
\end{defn}
We can also introduce the notion of connection and curvature on formal line bundles. 
\begin{lem}\label{lemma: formal-holomorphic-line-bundle-correspondence-1-1-class}
For every formal closed $(1,1)$-form $\alpha\in \A_{closed}^{1,1}(X)[[\hbar]]$, there is a formal line bundle $L_{\alpha/\hbar}$ with connection $\nabla_{L_{\alpha/\hbar}}$ whose curvature is given by $\frac{1}{\hbar}\cdot\alpha$.  
\end{lem}
\begin{proof}
Similar to holomorphic line bundles, we choose a fine cover $\{U_i\}$ of $X$.
On each $U_i$, we choose a local trivialization $e_i$ of $L_{\alpha/\hbar}$ on $U$, and $f_i\in C^\infty(U_i)[[\hbar]]$ such that $\bar{\partial}\partial(f_i)=\alpha|_{U_i}$.
On each non-empty intersection $U_i\cap U_j$, since $f_i,f_j$ are both potentials of $\alpha|_{U_i\cap U_j}$, their difference must be a sum of holomorphic and anti-holomorphic functions:
$$
f_i-f_j=f_{ij}(z)+g_{ij}(\bar{z}),
$$
and the functions $f_{ij}$ and $g_{ij}$ are unique up to constants (i.e., elements in $\C[[\hbar]]$). 
So $e^{f_{ij}(z)/\hbar}\in\mathcal{O}_{X,e}(U_i\cap U_j)$ are well-defined functions satisfying the cocycle condition:
$$
e^{\frac{f_{ij}(z)+f_{jk}(z)+f_{ki}(z)}{\hbar}}=1\in\mathcal{O}_{X,e}(U_i\cap U_j\cap U_k).
$$
This defines the desired invertible $\mathcal{O}_{X,e}$-module $L_{\alpha/\hbar}$.
It is equipped with a connection $\nabla_{L_{\alpha/\hbar}}$ which acts locally as 
$$
\nabla_{L_{\alpha/\hbar}}(f\cdot e^{g/\hbar}\otimes e_i)=(df+\frac{1}{\hbar}(f\cdot dg+f\cdot\partial f_i))\cdot e^{g/\hbar}\otimes e_i. 
$$
It is easy to see that the connection is well-defined and has curvature $\nabla_{L_{\alpha/\hbar}}^2=\frac{1}{\hbar}\cdot\alpha$. 
\end{proof}

Let us explain the motivations for introducing the notion of formal line bundles.
From the point of view of the Berezin-Toeplitz quantization, the above lemma implies the existence of a formal line bundle whose curvature is $\omega/\hbar$. Since $1/\hbar\sim k$, this formal line bundle corresponds to the asymptotics of the tensor powers $L^{\otimes k}$ of the prequantum line bundle as $k\rightarrow\infty$. This also explain why we use the name ``formal'' line bundle. On the other hand, from the point of view of Fedosov's construction, we need to twist $\W_{X,e}$ by a formal line bundle so to admit a Fedosov flat connection
\begin{rmk}
 The notion of formal line bundle here is similar to that of {\em local line bundle} in \cite{Melrose} and {\em twisted bundle} in \cite{Tsygan}. All these originate from the same motivation, namely, to introduce geometric objects whose curvature can be defined when there is no integrality condition. A significant difference is that our formal line bundle encodes the complex structure, so it gives a natural generalization of the notion of prequantum line bundles on K\"ahler manifolds. 
\end{rmk}

We now explain the Fedosov viewpoint in more details. The idea is very simple: the Fedosov connection on the Weyl bundle $\W_{X,\C}$ is the sum of the Levi-Civita connection $\nabla$ and the bracket $\frac{1}{\hbar}[\gamma_\alpha,-]_\star$. A naive guess is to replace the bracket by an extended fiberwise Bargmann-Fock action $\frac{1}{\hbar}\gamma_\alpha\circledast$ on $\W_{X,e}$. We need to be careful here: for a monomial $f$ as in Definition \ref{definition: Fock-representation-of-Wick-algebra}, we can extend its action to $\W_{X,e}$ by the same differential operator as in equation \eqref{equation: Toeplitz-operators-C-n}. In particular, 
$$
\bar{y}^j\circledast(f\cdot e^{g/\hbar})=\hbar\frac{\omega^{i\bar{j}}}{2\sqrt{-1}}\frac{\partial}{\partial y^i}(f\cdot e^{g/\hbar})=\hbar\frac{\omega^{i\bar{j}}}{2\sqrt{-1}}\left(\frac{\partial f}{\partial y^i}+\frac{1}{\hbar}f\cdot\frac{\partial g}{\partial y^i}\right)\cdot e^{g/\hbar}
$$
However, for general elements of $\W_{X,\C}$ on $\W_{X,e}$ we could run into infinite sums such as the following example: when $X=\C$, $g=y$ and $f=\sum_{k\geq 1}\bar{y}^k$, 
\begin{align*}
f\circledast e^{g/\hbar}=\left(\sum_{k\geq 1}\bar{y}^k\right)\circledast e^{y/\hbar}
=\sum_{k\geq 1}\left(\hbar\partial_{y}\right)^k (e^{y/\hbar})
=e^{y/\hbar}+e^{y/\hbar}+\cdots.
\end{align*}
If we write $f=\sum_{k,I,J}\hbar^k f_{I,\bar{J},k}y^I\bar{y}^J$ and $g=\sum_{I,k}\hbar^k g_{k,I}y^I$, then it is not difficult to see from the above example that the infinite sums come from two sources:
\begin{enumerate}
 \item Those terms $g_{0,i}y^i$ in $g$ which are linear in $\W_{X}$ and do not include $\hbar$;
 \item The infinite sums $\sum_{J}f_{k_0,I_0,\bar{J}}\hbar^{k_0}y^{I_0}\bar{y}^J$ for fixed indices $I_0, k_0$. 
\end{enumerate}
Thus we have the following lemma:
\begin{lem}\label{lemma: finite-condition-module-action}
We say that a section $\alpha = \sum_{k,I,J}\hbar^k\cdot\alpha_{k,I,\bar{J}}y^I\bar{y}^J$ of $\W_{X,\C}$ is {\em admissible} if it satisfies the following finiteness condition: for every fixed $I_0$ and $k_0$, $\sum_{J}\hbar^{k_0}\alpha_{k_0,I_0,\bar{J}}y^{I_0}\bar{y}^J$ is a finite sum. Then for any admissible $\alpha$ and for any section $s$ of $\W_{X,e}$, there is a well-defined $\alpha\circledast s$.
\end{lem}


\begin{defn}\label{definition: Fock-sheaf}
	For a representative $\alpha$ of a formal $(1,1)$-class $[\alpha]\in \hbar H^{1,1}_{dR}(X)[[\hbar]]$, let $\gamma_\alpha$ be as in equation \eqref{equation: Fedosov-connection-gamma-alpha}. Since $\gamma_\alpha$ is admissible as in Lemma \ref{lemma: finite-condition-module-action}, 
	\begin{equation*}\label{equation: Fedosov-connection-W-X-e}
	D_\alpha\left(f\cdot e^{g/\hbar}\right) := \left(\nabla+\frac{1}{\hbar}\gamma_\alpha\circledast \right)\left(f\cdot e^{g/\hbar}\right).
	\end{equation*} 
	defines a connection on $\W_{X,e}[\hbar^{-1}]$.
Here $\nabla$ denotes the naturally extended Levi-Civita connection on $\W_{X,e}$:
$$
\nabla(f\cdot e^{g/\hbar})=\nabla(f)\cdot e^{g/\hbar}\pm \left(\frac{1}{\hbar}f\cdot\nabla(g)\right)\cdot e^{g/\hbar},
$$
and $\circledast$ denotes the fiberwise Bargmann-Fock action of $\W_{X,\C}$ on $\W_{X,e}$.
\end{defn}
\begin{lem}\label{lemma: curvature-Levi-Civita-extended-holomorphic-Weyl}
	The curvature of $D_\alpha$ is given by
	$$D_\alpha^2 = \frac{1}{\hbar}\omega_\hbar - \text{Ric}_X,$$
	where $\text{Ric}_X = R_{i\bar{j}k}^kdz^i\wedge d\bar{z}^j$ is the Ricci form of $X$.
	In particular, the connection $D_\alpha$ on $\W_{X,e}$ is not flat.
\end{lem}
\begin{proof}
Let $f\cdot e^{g/\hbar}$ be a section of $\W_{X,e}$. Then
\begin{align*}
\nabla^2(f\cdot e^{g/\hbar}) & = \nabla((\nabla(f)+f\nabla(g/\hbar))\cdot e^{g/\hbar})\\
& = (\nabla^2f+f\nabla^2(g/\hbar))\cdot e^{g/\hbar}.
\end{align*}
On the other hand, we have
\begin{align*}
\frac{1}{\hbar}R_\nabla\circledast (f\cdot e^{g/\hbar})
& = \frac{1}{\hbar}(-2\sqrt{-1})R_{i\bar{j}k}^m\omega_{m\bar{l}}dz^i\wedge d\bar{z}^j\otimes y^k\bar{y}^l\circledast (f\cdot e^{g/\hbar})\\
& = \frac{1}{\hbar}(-2\sqrt{-1})R_{i\bar{j}k}^m\omega_{m\bar{l}}dz^i\wedge d\bar{z}^j\otimes\hbar\frac{\omega^{p\bar{l}}}{2\sqrt{-1}}\frac{\partial}{\partial y^p}(y^k f\cdot e^{g/\hbar})\\
& = R_{i\bar{j}k}^mdz^i\wedge d\bar{z}^j\otimes \frac{\partial}{\partial y^m}(y^k f\cdot e^{g/\hbar})\\
& = R_{i\bar{j}k}^mdz^i\wedge d\bar{z}^j\otimes y^k\frac{\partial}{\partial y^m} (f\cdot e^{g/\hbar})+R_{i\bar{j}k}^kdz^i\wedge d\bar{z}^j\cdot (f\cdot e^{g/\hbar})\\
& = \nabla^2 (f\cdot e^{g/\hbar})+R_{i\bar{j}k}^kdz^i\wedge d\bar{z}^j\cdot (f\cdot e^{g/\hbar}).
\end{align*}
So the curvature of the Levi-Civita connection $\nabla$ on $\W_{X,e}$ is given by
\begin{equation*}\label{equation: curvature-Levi-Civita-extended-holomorphic-Weyl}
\nabla^2=\left(\frac{1}{\hbar}R_\nabla-R_{i\bar{j}k}^kdz^i\wedge d\bar{z}^j\right)\circledast.
\end{equation*}
Now we compute
\begin{align*}
\left(\nabla+\frac{1}{\hbar}{\gamma_\alpha}\circledast\right)^2(s)
& = \nabla^2(s)+\frac{1}{\hbar}\nabla(\gamma_\alpha\circledast s)+\frac{1}{\hbar}\gamma_\alpha\circledast(\nabla s+\frac{1}{\hbar}\gamma_\alpha\circledast s)\\
& = \nabla^2(s)+\frac{1}{\hbar}(\nabla\gamma_\alpha+\frac{1}{\hbar}\gamma_\alpha\star\gamma_\alpha)\circledast s\\
& = \frac{1}{\hbar}(\nabla\gamma_\alpha+\frac{1}{\hbar}\gamma_\alpha\star\gamma_\alpha+R_\nabla-\hbar R_{i\bar{j}k}^kdz^i\wedge d\bar{z}^j)\circledast s\\
& =\frac{1}{\hbar}(\omega_\hbar-\hbar\cdot R_{i\bar{j}k}^kdz^i\wedge d\bar{z}^j)\cdot s.
\end{align*}
\end{proof}

\subsection{Bargmann-Fock sheaves}\label{subsection: Bargmann-Fock-module-sheaf}
\

\begin{defn}\label{definition: Bargmann-Fock-sheaf}
	For a representative $\alpha$ of a formal $(1,1)$-class $[\alpha]\in \hbar H^{1,1}_{dR}(X)[[\hbar]]$, let $\omega_\hbar := 2\sqrt{-1}\cdot\omega-\alpha$ and $\alpha' := -\omega_\hbar+\hbar\cdot\text{Ric}_X$. Then we define the {\em sheaf of Bargmann-Fock modules} as
	$$\mathcal{F}_{X,\alpha} := \W_{X,e}\otimes_{\mathcal{O}_{X,e}} L_{\alpha'/\hbar}.$$
	It is equipped with the connection
	$$D_{B,\alpha} := (\nabla+\frac{1}{\hbar}\gamma_{\alpha}\circledast-)\otimes 1+1\otimes \nabla_{L_{\alpha'/\hbar}}.$$
\end{defn} 

Lemmas \ref{lemma: curvature-Levi-Civita-extended-holomorphic-Weyl} and \ref{lemma: formal-holomorphic-line-bundle-correspondence-1-1-class} imply that $D_{B,\alpha}$ is flat, i.e., $D_{B,\alpha}^2=0$.
We have seen that there is a well-defined Bargmann-Fock action of admissible sections in $\W_{X,\C}$ on $\mathcal{F}_{X,\alpha}$. In fact, this action is compatible with the connections on these sheaves:
\begin{lem}\label{lemma: compatibility-Fedosov-connection} 
The connection $D_{B,\alpha}$ is compatible with the Fedosov connection $D_{F,\alpha}$, i.e., if $O\in\W_{X,\C}$ is an admissible section and $s$ is a section of $\mathcal{F}_{X,\alpha}$ , then we have
\begin{equation*}\label{equation: compatibility-Fedosov-connection}
D_{B,\alpha}(O\circledast s)=D_{F,\alpha}(O)\circledast s+(-1)^{|O|} O\circledast(D_{B,\alpha}(s)).
\end{equation*}
In particular, if $O$ and $s$ are flat sections, then $O \circledast s$ is also a flat section.
\end{lem}
\begin{proof}
This actually comes from the compatibility between $D_\alpha$ and $D_{F,\alpha}$, namely, for $O\in\W_{X,\C}$ and $s \in \W_{X,e}$, we have
\begin{align*}
D_{\alpha}(O\circledast s) 
& = \left(\nabla+\frac{1}{\hbar}\gamma_\alpha\circledast\right)(O\circledast s)\\
& = \nabla(O)\circledast s+(-1)^{|O|} O\circledast\nabla(s)+\frac{1}{\hbar}(\gamma_\alpha\star O)\circledast s\\
& = \nabla(O)\circledast s+(-1)^{|O|}O\circledast\nabla(s)+\frac{1}{\hbar}[\gamma_\alpha,O]_\star\circledast s+(-1)^{|O|} O \circledast\left(\frac{1}{\hbar}\gamma_\alpha\circledast s\right)\\
& = D_{F,\alpha}(O)\circledast s+(-1)^{|O|}O\circledast D_{\alpha}(s).
\end{align*}
\end{proof}
Consider a smooth function $f$ with Taylor-Fedosov series $O_f=\sum_{I,J} a_{I\bar{J}}y^I\bar{y}^J$. If the action of $O_f$ on $\mathcal{F}_{X,\alpha}$ is well-defined, then the above compatibility implies that this action preserves flat sections under $D_{B,\alpha}$. 
However, we will need certain convergence property or analyticity in order to define an action of $\W_{X,\C}$ on $\mathcal{F}_{X,\alpha}$, since, as we mentioned right before Lemma \ref{lemma: finite-condition-module-action}, there will be certain infinite sums as in the following example:
\begin{align*}
\left(\sum_{J}a_{0,\bar{J}}\bar{y}^J\right)\circledast e^{\beta_iy^i/\hbar}
=&\left(a_{0,\bar{J}}\hbar^{|J|}\frac{\partial^{J}}{\partial y^{j_1}\cdots y^{j_{|J|}}}\right)(e^{\beta_iy^i/\hbar})\\
=&\sum_{J}a_{0,\bar{J}}\beta_{j_1}\cdots\beta_{j_{|J|}}\cdot e^{\beta_iy^i/\hbar}.
\end{align*}
From now on we will assume that the $(1,1)$-form $\omega_\hbar=2\sqrt{-1}\cdot\omega-\alpha$ is real analytic, and we define a function on $X$ which measures its analyticity:
\begin{lem-defn}
 Suppose $\omega_\hbar=2\sqrt{-1}\cdot\omega-\alpha$ is real analytic. We define a function $r: X\rightarrow(0,\infty]$ by letting $r(x_0)$ be the radius of convergence of $\omega_\hbar$ under a $K$-coordinate centered at $x_0$. The function $r$ is lower semi-continuous, and is independent of the choice of  $K$-coordinates because different choices differ only by a $U(n)$ transformation. Equivalently, suppose $r(x_0)>r_0$, then there exists a neighborhood $x_0\in U$, such that $r(x)>r_0$ for all $x\in U$. 
\end{lem-defn}

We define the following sub-class of real analytic functions, which roughly speaking consists of analytic functions with analyticity at least the same as that of $\omega_\hbar$.
\begin{defn}\label{definition: convergence-property-analytic-function}
For every open set $U\subset X$, let $C^{\omega_\hbar}_X(U)[[\hbar]]$ denote the set of real analytic functions on $X$ such that at every point $x_0 \in X$, the radius of convergence is greater than or equal to $r(x_0)$ under a $K$-coordinate centered at $x_0$.  
\end{defn}


\begin{lem}
 The spaces $C^{\omega_\hbar}_X(U)[[\hbar]]$ of functions define a sheaf $C^{\omega_\hbar}_X[[\hbar]]$ of algebras on $X$ under the Fedosov star product $\star_\alpha$. 
\end{lem}


\begin{proof}
Since this convergence property of functions is defined pointwise, it is clear that the spaces define a sub-sheaf of the sheaf of real analytic functions. For every point $x_0\in X$, we fix a $K$-coordinate centered at $x_0$. From the Fedosov construction of the Wick type deformation quantization, we can see that the coefficients of the bi-differential operators $C_i(-,-)$ are either the Christoffel symbols of the Levi-Civita connection, the coefficients of $\omega_\hbar$ or their derivatives, which all have at least the same convergence property as the formal closed $(1,1)$-form $\omega_\hbar$.
\end{proof}
\begin{defn}\label{definition: flat-Bargmann-Fock}
The {\em Bargmann-Fock sheaf} $\mathcal{F}_{X,\alpha}^{\text{flat}}$ is defined as the sub-sheaf of $\mathcal{F}_{X,\alpha}$ which consists of flat sections that are finite sums of the following form: $\alpha\cdot e^{\beta/\hbar}\otimes e_U$, where we can write $\beta=\sum_{|I|\geq 0}\beta_Iy^I$ locally.
We require that the coefficients of the degree $1$ terms, i.e., $\beta_i$, $1\leq i\leq n$, satisfy the following boundedness condition:
\begin{equation}\label{equation: convergent-Fock-boundedness-condition}
\left\|\sum_{i=1}^n\beta_iy^i\right\|_{x_0} < r(x_0),
\end{equation}
where the norm is defined using the Hermitian metric on $T^*X$. 
\end{defn}
Notice that this definition is independent of the local trivializations of the holomorphic Weyl bundle and the formal line bundle. 

\begin{lem}\label{lemma: anti-holomorphic-Taylor-series-convergent-radius}
Consider $f\in C^{\omega_\hbar}(X)$ with Taylor-Fedosov series given locally by $O_f=\sum_{I,J}a_{I\bar{J}}y^I\bar{y}^J$ under a $K$-coordinate centered at $x_0$. For every multi-index $I_0$, the series
\begin{equation}\label{equation: convergence-fixed-holomorphic-degree}
\left(\sum_{J}a_{I_0\bar{J}}y^{I_0}\bar{y}^J\right)\Bigg|_{z=x_0,(y^1,\cdots, y^n)=(\xi^1,\cdots, \xi^n)},
\end{equation}
\begin{equation}\label{equation: convergence-fixed-holomorphic-degree-derivatives}
\nabla\left(\sum_{J}a_{I_0\bar{J}}y^{I_0}\bar{y}^J\right)\Bigg|_{z=x_0,(y^1,\cdots, y^n)=(\xi^1,\cdots, \xi^n)} 
\end{equation}
converge for $\xi=(\xi_1,\cdots,\xi_n)\in\C^n$ with $||\xi||<r(x_0)$.
\end{lem}
\begin{proof}
We can apply the following iterative equation for the Taylor-Fedosov series of $f$:
$$
O_f=f+\delta^{-1}(\nabla O_f+\frac{1}{\hbar}[I_\alpha,O_f]_\star).
$$
A simple observation is that all terms in $\delta^{-1}\circ\left(\nabla+\frac{1}{\hbar}[I_\alpha,-]_\star\right)$, except $\delta^{-1}\circ\nabla^{0,1}$, will increase either the holomorphic degree in $\W_{X,\C}$ or the degree of $\hbar$. Thus all but finitely many terms in $\sum_{J}a_{I_0\bar{J}}y^{I_0}\bar{y}^J$ can be obtained by applying the operator $\delta^{-1}\circ\nabla^{0,1}$ a number of times to a function which is a product of $f$, the Christoffel symbols, coefficients of $\omega_\hbar$ and the curvature tensor, and also their derivatives. These terms are exactly the purely anti-holomorphic part of the Taylor expansion of these functions under the $K$-coordinates at $x_0$, and the convergence of the series \eqref{equation: convergence-fixed-holomorphic-degree} follows. 

The proof for the convergence of the series \eqref{equation: convergence-fixed-holomorphic-degree-derivatives} is similar. From the discussion in the previous paragraph, we can assume, without loss of generality, that $I_0=0$ and $\sum_{J}a_{I_0\bar{J}}y^{I_0}\bar{y}^J=\sum_{k\geq 0}(\tilde{\nabla}^{0,1})^kg$, where $g$ is a function which is a product of $f$, the Levi-Civita connection, coefficients of $\omega_\hbar$, and their derivatives. We can split the series \eqref{equation: convergence-fixed-holomorphic-degree-derivatives} into its $(1,0)$- and $(0,1)$-parts. It is then clear that the $(0,1)$-part has the desired convergence property. Thus we only need to show that 
$$
\sum_{k\geq 0}\nabla^{1,0}\circ(\tilde{\nabla}^{0,1})^{k}(g)\Big|_{z=x_0,(y^1,\cdots, y^n)=(\xi^1,\cdots, \xi^n)} 
$$
converges for $||\xi||_{x_0}\leq r(x_0)$.  Notice that for every fixed $k\geq 0$, we have
\begin{align*}
\left(\nabla^{1,0}\circ(\tilde{\nabla}^{0,1})^{k}\right)(g)
= & [\nabla^{1,0},\tilde{\nabla}^{0,1}]\circ(\tilde{\nabla}^{0,1})^{k-1}(g)\pm\tilde{\nabla}^{0,1}\circ[\nabla^{1,0},\tilde{\nabla}^{0,1}]\circ(\tilde{\nabla}^{0,1})^{k-2}(g)+\cdots\\
&\quad \pm(\tilde{\nabla}^{0,1})^{k-1}\circ[\nabla^{1,0},\tilde{\nabla}^{0,1}]\pm(\tilde{\nabla}^{0,1})^{k}\circ\nabla^{1,0}(g),
\end{align*}
and the bracket $[\nabla^{1,0},\tilde{\nabla}^{0,1}]$ contributes the coefficients of the curvature. Thus, when evaluated at $z=x_0, y=\xi$, the absolute values of these terms are bounded by $k$ times the components of the Taylor series of functions with radius of convergence at least $r(x_0)$. Now the convergence of \eqref{equation: convergence-fixed-holomorphic-degree-derivatives} follows. 
\end{proof}

We also need the following lemma from elementary analysis:
\begin{lem}\label{lemma: convergence-of-derivatives-series}
 Let $f_n: U\rightarrow\C, n\in\N$ be a sequence of smooth functions on $U$,  and $D$ be a differential operator on $U$. Suppose that the two series $\sum_{i=1}^\infty f_n$ and $\sum_{i=1}^\infty D(f_n)$ converge uniformly to $S:U\rightarrow\C$ and $g:U\rightarrow\C$ respectively. 
Then $g=D(S)$. In other words, the infinite sum commutes with the differential operator $D$.  
\end{lem}

\begin{thm}\label{theorem: Bargmann-Fock-representation}
 The Bargmann-Fock sheaf $\mathcal{F}_{X,\alpha}^{\text{flat}}$ is a sheaf of modules over $\left(C^{\omega_\hbar}_X[[\hbar]],\star_\alpha\right)$. 
\end{thm}
\begin{proof}
Let $O_f=\sum_{I,J}a_{I,\bar{J}}y^I\bar{y}^J$ be the Taylor-Fedosov series of a function $f\in C^{\omega_\hbar}_X(U)[[\hbar]]$.  We only need to construct a well-defined $O_f\circledast e^{\beta/\hbar}$ with $\beta=\beta_Iy^I$ and show that $D_{B,\alpha}(O_f\circledast e^{\beta/\hbar})=0$. For every $x_0\in U$, we first prove the convergence of $\left(O_f(e^{\beta/\hbar})\right)|_{x_0}$. We choose a $K$-coordinate centered at $x_0$, and assume, without loss of generality, that the fixed holomorphic index $I_0=0$. Then the inequality \eqref{equation: convergent-Fock-boundedness-condition} becomes $||(\beta_1(x_0),\cdots,\beta_n(x_0))||\leq r(x_0)$ under the standard norm on $\mathbb{C}^n$, and we obtain the following infinite sum:
\begin{equation*}\label{equation: infinite-sum-after-Bargmann-Fock-action}
\begin{aligned}
\left(a_{0,J}\bar{y}^J\right)(e^{\beta_iy^i/\hbar})\big|_{x_0}
& = \left(a_{0,J}\hbar^{|J|}\frac{\partial^{J}}{\partial y^{j_1}\cdots y^{j_{|J|}}}\right)(e^{\beta_iy^i/\hbar})\big|_{x_0}\\
& = \sum_{J}a_{0,J}(x_0)\beta_{j_1}(x_0)\cdots\beta_{j_{|J|}}(x_0)\cdot e^{\beta_iy^i/\hbar}\big|_{x_0}.
\end{aligned}
\end{equation*}
The absolute convergence of this series follows from the inequality \eqref{equation: convergent-Fock-boundedness-condition} and Lemma \ref{lemma: anti-holomorphic-Taylor-series-convergent-radius}.  
Now convergence of the series \eqref{equation: convergence-fixed-holomorphic-degree-derivatives} and Lemma \ref{lemma: convergence-of-derivatives-series} imply that the connection $D_{B,\alpha}$ commutes with the infinite sum, thus we have $D_{B,\alpha}(O_f\circledast e^{\beta/\hbar})=0$.
\end{proof}

\section{Formal Berezin-Toeplitz quantizations}\label{section: micri-local-Berezin-Toeplitz}

In this section, we consider the case when $X$ is pre-quantizable and equipped with a prequantum line bundle $L$. We show that the Bargmann-Fock sheaf in this situation gives rise to a micro-local description of the asymptotics of Toeplitz operators on holomorphic sections of $L^{\otimes k}$. This can be generalized to all Wick type star products on K\"ahler manifolds whose Karabegov forms are real analytic. 

\subsection{Formal Toeplitz operators}
\


We first give a brief review of formal Hilbert spaces and the associated formal Toeplitz operators defined in \cite{Chan-Leung-Li}.  First of all, the Wick algebra has the following analytic interpretation. In the flat case when $X = \mathbb{C}^n$ (and with trivial prequantum line bundle $L$), the Hilbert space on which the Toeplitz operators act is the well-known {\em Bargmann-Fock space} $\mathcal{H}L^2(\mathbb{C}^n,\mu_\hbar)$, which consists of $L^2$ integrable entire holomorphic functions with respect to the density $\mu_\hbar(z) = (\pi\hbar)^{-n}e^{-|z|^2/\hbar}$; here $\hbar$ is regarded as a positive real number. 

It is easy to see, by direct computations, that the holomorphic polynomials 
$$
\frac{z^I}{\sqrt{I!\hbar^{|I|}}},
$$
where $I$ runs over all multi-indices, form an orthonormal basis of $\mathcal{H}L^2(\mathbb{C}^n,\mu_\hbar)$. Toeplitz operators associated to polynomials are defined by multiplying by a polynomial $f\in\C[z,\bar{z}]$, which is in general non-holomorphic, and then projecting back to the holomorphic subspace. For example, when $n=1$, we have
\begin{align*}
T_z = m_z,\quad
T_{\bar{z}} = \hbar\frac{d}{dz},\quad 
T_{f_1(z)f_2(\bar{z})} = f_2\left(\hbar\frac{d}{dz}\right)\circ m_{f_1(z)},
\end{align*}
For any $f,g\in\mathbb{C}[z,\bar{z}]$, we have $T_f\circ T_g=T_{f\star g}$.

By regarding $\hbar$ as a formal variable instead, we can interpret the $T_f$'s as Toeplitz operators on $\W_{\C^n}$, where the formal inner product is defined using Feynman graph expansions:
$$
\langle f, g\rangle:=\frac{1}{\hbar^n}\cdot\int f\bar{g}\cdot e^{\frac{-|y|^2}{\hbar}}\in\C[[\hbar]].
$$
More generally, we may allow perturbations of the Gaussian measure $e^{-|y|^2/\hbar}$ by {\em interaction terms}, and define formal Hilbert spaces: 
\begin{defn}
	Suppose that all the terms in $\phi(y,\bar{y})\in\W_{\mathbb{C}^n}$ have weight at least $3$. Then for $f, g\in\mathcal{W}_{\mathbb{C}^n}((\sqrt{\hbar}))$, we define their {\em formal inner product} as the formal integral
	\begin{equation}\label{defn:formal-inner-product}
		\langle f, g\rangle:=\frac{1}{\hbar^n}\cdot\int f\bar{g}\cdot e^{\frac{-|y|^2+\phi(y,\bar{y})}{\hbar}} \in \mathbb{C}((\sqrt{\hbar})),
	\end{equation}
	which is in turn defined using Feynman graph expansions. 
\end{defn}
\begin{defn}\label{defn:formal-Toeplitz-operator}
	The {\em orthogonal projection}
	\begin{equation*}\label{equation: formal-Toeplitz-operator}
		\pi_{\phi}: \mathcal{W}_{\mathbb{C}^n}\rightarrow \mathcal{F}_{\mathbb{C}^n}=\mathbb{C}[[y_1,\cdots, y_n]][[\hbar]]
	\end{equation*}
	is defined by requiring that
	$$
	\langle f, y^I\rangle=\langle\pi_{\phi}(f), y^I\rangle
	$$
	for all multi-indices $I$; here $\langle-,-\rangle$ is the formal inner product defined in \eqref{defn:formal-inner-product}.
	
	The {\em formal Toeplitz operator} $T_{f,\phi}$ associated to $f\in\mathcal{W}_{\mathbb{C}^n}$ is defined as the composition of multiplication by $f$ and the orthogonal projection $\pi_{\phi}$:
	$$
	T_{f,\phi}:=\pi_{\phi}\circ m_{f}.
	$$
\end{defn}
\begin{prop}[Theorem 2.2 and its proof  in \cite{Chan-Leung-Li}]\label{proposition: formula-formal-Toeplitz} 
For any $f\in\mathcal{W}_{\mathbb{C}^n}$, there exists a unique $O_f$ in the Wick algebra such that
\begin{enumerate}
  \item $T_{f,\phi}(s)=O_f\star s$, for any $s\in\mathcal{F}_{\mathbb{C}^n}$;
  \item Let $f$ be a monomial, then the leading term of $O_f$ is exactly $f$, i.e.,
         $$
         O_f=f+\cdots,
         $$
         where the dots denote terms of degree greater than $\deg(f)$. 
  \item $O_f$ is the unique solution of the following equation:
  \begin{equation}\label{equation: formal-Toeplitz-operator-equality}
  f\cdot e^{\phi/\hbar}=e^{\phi/\hbar}\star O_f.
  \end{equation}
 \end{enumerate}
\end{prop}

\subsection{Prequantum Bargmann-Fock sheaves}\label{section: prequantum-Bargmann-Fock}
\

In this subsection, we consider the following sheaf of Bargmann-Fock modules:
$$
\mathcal{F}_{X, \alpha} = \W_{X,e}\otimes_{\mathcal{O}_{X,e}}L_{-2\sqrt{-1}\omega/\hbar},
$$
which we call a {\em prequantum Bargmann-Fock sheaf}; here $\alpha = -\hbar \text{Ric}_X =  -\hbar\cdot R_{i\bar{j}k}^kdz^i\wedge d\bar{z}^j$. When $X$ is pre-quantizable, we will explain in the next subsection how this prequantum Bargmann-Fock sheaf describes the micro-local behavior of the Berezin-Toeplitz operators on holomorphic sections of $L^{\otimes k}$.

For any local holomorphic frame $e_{x_0}$ of $L_{-2\sqrt{-1}\omega/\hbar}$ around $x_0$ such that
\begin{equation}\label{equation: connection-formal-line-bundle}
\nabla_{L_{-2\sqrt{-1}\omega/\hbar}}(e_{x_0})=-\frac{1}{\hbar}\partial\rho\otimes e_{x_0},
\end{equation}
where $\rho$ is a K\"ahler potential (i.e. $\partial\bar{\partial}(\rho)=-2\sqrt{-1}\omega$), 
we define a local section of the holomorphic Weyl bundle $\W_X$ by $\beta=\sum_{k\geq 1}(\tilde{\nabla}^{1,0})^k(\rho)$.
The following theorem, whose proof will be given in Appendix \ref{appendix: flat-section-Bargmann-Fock}, describes some local flat sections of the prequantum Bargmann-Fock sheaf:
\begin{thm}\label{theorem: flat-section-Bargmann-Fock}
Suppose a section of the prequantum Bargmann-Fock sheaf is of the form $A\cdot e^{\beta/\hbar}\otimes e_{x_0}$ around $x_0$, where $A$ is a section of $\W_X$. Then $D_{B,\alpha}\left(A\cdot e^{\beta/\hbar}\otimes e_{x_0}\right)=0$ if and only if $D_K(A)=0$, or equivalently, $A=J_s$ for some holomorphic function $s$. 
\end{thm}

\begin{prop}\label{proposition: local-section-prequantum-Bargmann-Fock}
Suppose the holomorphic frame $e_{x_0}$ in Equation \eqref{equation: connection-formal-line-bundle} is chosen such that $\rho=\rho_{x_0}$ in Definition \ref{definition: K-coordinates-K-frame}. There exists a neighborhood $U$ of $x_0$, such that
\begin{equation}\label{equation: canonical-flat-section-prequantum-sheaf}
e^{\beta/\hbar}\otimes e_{x_0}\in\mathcal{F}_{X,\alpha}^{\text{flat}}(U).
\end{equation}
\end{prop}
\begin{proof}
By Theorem \ref{theorem: flat-section-Bargmann-Fock}, $D_{B,\alpha}\left(e^{\beta/\hbar}\otimes e_{x_0}\right)=0$. Thus we only need to prove that the boundedness condition \eqref{equation: convergent-Fock-boundedness-condition} is satisfied. From the Taylor expansion \eqref{equation: Taylor-expansion-hermitian-metric} of $\rho_{x_0}$ and the definition of $\beta$, we have $\beta|_{x_0}=0$. In particular,  
$\beta_iy^i|_{x_0}=0$.
Thus the lower semi-continuity of $r(x)$ implies that the condition \eqref{equation: convergent-Fock-boundedness-condition} holds in a neighborhood $U$ of $x_0$.  
\end{proof}
\begin{rmk}
Proposition \ref{proposition: local-section-prequantum-Bargmann-Fock} shows that for a small enough open set $U\subset X$, the space $\mathcal{F}_{X,\alpha}^{\text{flat}}(U)$ is not empty. This can be easily generalized to  general Bargmann-Fock sheaves. 
\end{rmk}

\begin{prop}\label{proposition: representation-stalk}
There exists a subspace $V_{x_0}$ of the stalk $(\mathcal{F}_{X,\alpha}^{\text{flat}})_{x_0}$, which is isomorphic to the space of germs of formal holomorphic functions at $x_0$, i.e., $V_{x_0}\cong\mathcal{O}_{X,x_0}[[\hbar]]$, such that for every neighborhood $U$ of $x_0$, $V_{x_0}$ is a representation of $(C^{\omega}(U)[[\hbar]],\star)$. 
\end{prop}
\begin{proof}
  Let $f\in C^{\omega}(U)$, then $O_f\circledast\left(J_s\cdot e^{\beta/\hbar}\otimes e_{x_0}\right)$ must be of the form $A\cdot e^{\beta/\hbar}\otimes e_{x_0}\in\mathcal{F}_{X,\alpha}^{\text{flat}}(U)$. Theorem \ref{theorem: flat-section-Bargmann-Fock} then implies that $A=J_{s'}$ for some $s'\in\mathcal{O}_X(U)[[\hbar]]$. By looking at the germs of these sections of $\mathcal{F}_{X,\alpha}^{\text{flat}}$ at $x_0$, we obtain the representation. 
\end{proof}

\begin{rmk}
 In a subsequent paper, we will show that if a K\"ahler manifold is prequantizable, then the formal variable $\hbar$ in the Bargmann-Fock sheaf can be replaced by $1/k$ for any positive integer $k$. Flat sections are then in a one-to-one correspondence to the holomorphic sections $H^0(X,L^{\otimes k})$. This is a generalization of Fedosov's original construction from quantum algebras to their modules. 
\end{rmk}


\subsection{Star products as formal Toeplitz operators}\label{section: star-product-formal-Toeplitz}
\

Suppose that the K\"ahler manifold $X$ is pre-quantizable with the prequantum line bundle given by $(L,\nabla_L)$.
We fix a K-coordinate system $(z_1,\cdots, z_n)$ at $x_0\in X$, and also a holomorphic frame $e_{L,x_0}$ of $L$ around $x_0$ satisfying $-\log||e_{L,x_0}||^2=\rho_{x_0}$ (we call this the {\em K-frame}). The local frame $e_{x_0}$ of $L_{\omega/\hbar}$ is the asymptotics of the local frame $e_{L,x_0}^{\otimes k}$ of $L^{\otimes k}$ as $k\rightarrow\infty$. The prequantum condition implies that
$$
\omega=\frac{\sqrt{-1}}{2\pi}\bar{\partial}\partial \log||e_{L,x_0}||^2
$$

We first recall the following proposition in Tian's paper \cite{Tian} and, in particular, the notion of {\em peak sections}.  We define the function $r(z) := \sqrt{|z_1|^2+\cdots+|z_n|^2}$.
\begin{prop}[Lemma 1.2 in \cite{Tian}]\label{proposition: peak-section}
	For a multi-index $p=(p_1,\cdots,p_n)\in\mathbb{Z}_+^n$ and an integer $r>|p|=p_1+\cdots+p_n$, there exists $m_0>0$ such that, for $m>m_0$, there is a holomorphic global section $S$, called a {\em peak section}, of the line bundle $L^{\otimes m}$, satisfying 
	\begin{equation}\label{equation: property-peak-section-norm}
 		\int_X ||S||^2_{h^m}dV_g=1,\quad \int_{X\setminus\{r(z)\leq\frac{\log m}{\sqrt{m}}\}}||S||_{h^m}^2dV_g=O\left(\frac{1}{m^{2r}}\right),
	\end{equation}
	and locally at $x_0$ under the K-coordinates, 
	\begin{equation}\label{equation: Taylor-expansion-peak-section}
 		S(z)=\lambda_{m,p}\cdot\left(z_1^{p_1}\cdots z_n^{p_n}+O(|z|^{2r})\right)e_{L,x_0}^m\left(1+O\left(\frac{1}{m^{2r}}\right)\right),
	\end{equation}
	where $||\cdot||_{h^m}$ is the norm on $L^{\otimes m}$ given by $h^m$, and $O\left(\frac{1}{m^{2r}}\right)$ denotes a quantity dominated by $C/m^{2r}$ with the constant $C$ depending only on $r$ and the geometry of $X$, moreover
	\begin{equation}\label{equation: coefficient-leading-term-Taylor-expansion-peak-section}
 		\lambda^{-2}_{m,p}=\int_{r(z)\leq\log m/\sqrt{m}}|z_1^{p_1}\cdots z_n^{p_n}|^2\cdot e^{-m\cdot\rho_{x_0}(z)} dV_g,
	\end{equation}
	where $dV_g=\frac{\omega^n}{n!}=(\sqrt{-1})^n\cdot h(z,\bar{z})\cdot dz^1\wedge d\bar{z}^1\wedge\cdots \wedge dz^n\wedge d\bar{z}^n$ is the volume form. 
\end{prop}
We normalize the peak section $S(z)$ in Proposition \ref{proposition: peak-section} and define 
$$
S_{m,p,r}:=\frac{1}{\lambda_{m,p}}\cdot\left(1+O\left(\frac{1}{m^{2r}}\right)\right)\cdot S=\left(z_1^{p_1}\cdots z_n^{p_n}+O(|z|^{2r})\right)\cdot e_{L,x_0}^m
$$
These normalized peak sections $S_{m,p,r}$ are roughly speaking, global holomorphic sections of $L^{\otimes m}$ such that
\begin{enumerate}
 \item its norm is concentrated around a given point $x_0$ on the K\"ahler manifold, and
 \item modulo higher order terms, its Taylor expansion at $x_0$ is $z_1^{p_1}\cdots z_n^{p_n}$ with respect to the K-coordinates and K-frame at $x_0$.
\end{enumerate}

In \cite{Chan-Leung-Li}, we constructed a formal Hilbert space $H_{x_0}$ as a sub-quotient of the vector space generated by these normalized peak sections $S_{m,p,r}$, and prove that this gives a nice representation of the Berezin-Toeplitz deformation quantization algebra.
Let us briefly recall the construction here. Let $S_{m,p_1,r}, S_{m,p_2,r}$ be normalized peak sections of $L^{\otimes m}$ with $r>>|p_1|,|p_2|$, we can identify them with holomorphic functions $f_1, f_2$ with respect to the K-frame $e_{L,x_0}$. Then the following integral 
\begin{equation}\label{equation: normalized-inner-product}
m^n\cdot\int_X \langle s_1,s_2\rangle_{h^m}dV_g=m^n\cdot\int_Xf_1(z)\bar{f}_2(z)\cdot e^{-m\cdot\rho_{x_0}(z,\bar{z})}\cdot\left(\frac{\sqrt{-1}}{2}\right)^n h(z,\bar{z})dz^1 d\bar{z}^1\cdots dz^n d\bar{z}^n
\end{equation}
is also concentrated around $x_0$, where it is a Gaussian integral. Thus these integrals have asymptotics as $m\rightarrow\infty$ given by Feynman graph expansions, for which we need to know the Taylor expansions of $f_1, f_2,\rho_{x_0}$ and $\log(h(z,\bar{z}))$ at $x_0$. For the first three functions, their Taylor expansions at $x_0$ are given by $(J_{f_1})_{x_0}, (J_{f_2})_{x_0}$ and $(\Phi_\omega)_{x_0}$ respectively. 
\begin{lem}\label{lemma: local-volume-form-purely-holomorphic-derivatives-vanish}
 The purely (anti-) holomorphic derivatives of $h$ vanishes at $z_0$ under the K-coordinates:
$$
\frac{\partial^{|I|} h}{\partial z^I}(x_0)=\frac{\partial^{|J|} h}{\partial \bar{z}^I}(x_0)=0,
$$
for all mutli-indices with $|I|,|J|>0$, and  $h(x_0)=1$. 
\end{lem}
\begin{proof}
The equality $h(x_0)=1$ follows from the definition of K-coordinates. We will show the vanishing of purely holomorphic derivatives at $x_0$; the proof for antiholomorphic ones is the same. It suffices to show that the statement is valid for functions $\omega_{i_1\bar{j_1}}\cdots\omega_{i_n\bar{j_n}}$, where
$$
2\sqrt{-1}\cdot\omega_{i\bar{j}}=\frac{\partial^2\rho_{x_0}}{\partial z^i\partial \bar{z}^j},
$$
The Taylor expansion of $\rho_{x_0}$ in equation \eqref{equation: Taylor-expansion-hermitian-metric} implies that
$$
\frac{\partial^{|I|+1}\rho_{x_0}}{\partial z^I\partial\bar{z}^j}(x_0)=0,\ |I|\geq 2,
$$
from which the first statement follows easily.  
\end{proof}
Since $h(z,\bar{z})$ is the Hermitian metric on the anti-canonical line bundle $K_X$ induced by the K\"ahler structure under the frame $\partial_{z^1}\wedge\cdots\wedge \partial_{z^n}$, there is
$$
\partial\bar{\partial}\log(h(z,\bar{z}))=R_{i\bar{j}k}^kdz^i\wedge d\bar{z}^j.
$$
Let $\alpha:=-\hbar\cdot\partial\bar{\partial}(\log h(z,\bar{z}))=-\hbar\cdot R_{i\bar{j}k}^kdz^i\wedge d\bar{z}^j$. Then Lemma \ref{lemma: local-volume-form-purely-holomorphic-derivatives-vanish} implies that the Taylor expansion of $\log(h(z,\bar{z}))$ at $x_0$ under $K$-coordinates is exactly given by 
$$
\left(-\frac{1}{\hbar}\Phi_\alpha\right)_{x_0}.
$$
Let $\Phi:=|y|^2+2\sqrt{-1}\cdot(\Phi_\omega)_{x_0}-(\Phi_\alpha)_{x_0}$. 
Let $\hbar=1/m$, then the asymptotic of the integral \eqref{equation: normalized-inner-product} is given by the following formal integral: 
$$
\frac{1}{\hbar^n}\cdot\int (J_{f_1})_{x_0}\cdot (J_{\bar{f}_2})_{x_0}\cdot e^{(2\sqrt{-1}\cdot\Phi_{\omega}/\hbar+\Phi_{\partial\bar{\partial}(\log(h(z,\bar{z}))})_{x_0}}=\frac{1}{\hbar^n}\cdot\int (J_{f_1})_{x_0}\cdot (J_{\bar{f}_2})_{x_0}\cdot e^{\frac{-|y|^2+\Phi}{\hbar}}.
$$
By taking the asymptotics of $S_{m,p,r}$'s as $m,r\rightarrow\infty$, we can ignore the remainder terms and essentially get the monomial $z_1^{p_1}\cdots z_n^{p_n}$. This is roughly how we can define the formal Hilbert space:
$$
H_{x_0}\cong \C[[z_1,\cdots, z_n]][[\hbar]],
$$
equipped with formal inner product defined via the above formal integral.

The vector space $V_{x_0}$ in Proposition \ref{proposition: representation-stalk} is  naturally a subspace of $H_{x_0}$ consisting of those formal power series which are convergent in some neighborhood of $0\in\C^n$. 
We have the following theorem:
\begin{thm}\label{theorem: Bargmann-Fock-microlocal-Toeplitz}
The representation of $C^{\omega_\hbar}(X)[[\hbar]]$ on $V_{x_0}$ defined in Proposition \ref{proposition: representation-stalk} is given explicitly as follows: 
Let $f\in C^{\omega_\hbar}(X)[[\hbar]] $ and $\Psi_s:=J_s\cdot e^{\beta/\hbar}\otimes e_{x_0}\in V_{x_0}\subset(\mathcal{F}_{X,\alpha}^{\text{flat}})_{x_0}$ where $s\in\mathcal{O}_{X,x_0}[[\hbar]]$. Then 
$$
O_f\circledast\left(J_s\cdot e^{\beta/\hbar}\otimes e_{x_0}\right)=J_{s'}\cdot e^{\beta/\hbar}\otimes e_{x_0},
$$
where $s'\in\mathcal{O}_{X,x_0}[[\hbar]]$ is determined by its jets $J_{s'}$ at $x_0$ explicitly given by
$$
T_{(J_f)_{x_0},\Phi}\left(J_s\right) = J_{s'}.
$$
In other words,  $s^{\prime }$ is obtained from the formal Toeplitz operation using $f$ on $s$.
\end{thm}
\begin{proof}
It is easy to see that the action of $\W_{X,\C}$ on $\mathcal{F}_X$ is linear over smooth functions. On the other hand,  $J_{s'}$ determines $s'$ since it is the Taylor expansion of the holomorphic function $s'$ at $x_0$. We have seen that $\beta|_{x_0}=0$ in the proof of Proposition \ref{proposition: local-section-prequantum-Bargmann-Fock}. We have
$$
J_{s'}=(O_f)_{x_0}\circledast J_s.
$$
On the other hand, Theorem \ref{thm:main-thm-formal-Toeplitz-Fedosov-equation} says that
$$
J_f\cdot e^{\Phi/\hbar}=e^{\Phi/\hbar}\star O_f.
$$
Comparing with equation \eqref{equation: formal-Toeplitz-operator-equality}, the result follows. 
\end{proof}

It is not difficult to see from the above construction and computation that we do not need to assume that $X$ admits a prequantum line bundle. For every Wick type star product whose Karabegov form is real analytic, there is a subspace in the stalk of the Bargmann-Fock sheaf, similar to $V_{x_0}$ above, such that formal smooth functions act on as formal Berezin-Toeplitz operators. 




\appendix

\section{Proof of Theorem \ref{thm:main-thm-formal-Toeplitz-Fedosov-equation}}\label{appendix: proof-theorem-jets-flat-section}
Let $\Phi := 2\sqrt{-1}\left(-\omega_{i\bar{j}}y^i\bar{y}^j+\Phi_\omega\right)-\Phi_\alpha$. It is clear that $$e^{\Phi/\hbar} := 1+\Phi/\hbar+\frac{1}{2!}(\Phi/\hbar)^2+\cdots$$
is an invertible section in $\W_{X,\C}^+$ under the Wick product, and we denote by $\left(e^{\Phi/\hbar}\right)^{-1}$ its inverse.

\begin{lem}\label{lemma: Fedosov-connection-jet-connection-conjugate}
Let $O$ be any section of the Weyl bundle. Then we have
\begin{equation}\label{equation: Fedosov-connection-jet-connection-conjugate-0-1}
(\nabla^{0,1}-\delta^{0,1})\left(e^{\Phi/\hbar}\star O\star(e^{\Phi/\hbar})^{-1}\right)=e^{\Phi/\hbar}\star D_{F,\alpha}^{0,1}(O)\star(e^{\Phi/\hbar})^{-1}.
\end{equation}
In other words, the operators $D_{F,\alpha}^{0,1}$ and $\nabla^{0,1}-\delta^{0,1}$ differ by the gauge action by $e^{\Phi/\hbar}$. 
\end{lem}
\begin{proof}
 We will restrict our attention to the case where $\alpha=0$; the general case is similar. The operator $\left(\nabla^{0,1}-\delta^{0,1}\right)$ is a derivation with respect to both the classical and quantum product on $\W_{X,\C}$, there is 
 \begin{align*}
 &\left(\nabla^{0,1}-\delta^{0,1}\right)e^{\Phi/\hbar}\\
 =&\frac{1}{\hbar}\left(\nabla^{0,1}\Phi-\delta^{0,1}\Phi\right)\cdot e^{\Phi/\hbar}\\
 =&\frac{1}{\hbar}\sum_{k\geq 2}\left(\nabla^{0,1}\Phi_{*,k}-\delta^{0,1}\Phi_{*,k}\right)\cdot e^{\Phi/\hbar}\\
 =&-\frac{1}{\hbar}(\delta^{0,1}\Phi_{*,2})\cdot e^{\Phi/\hbar}+\frac{1}{\hbar}\sum_{k\geq 2}\left(\nabla^{0,1}\Phi_{*,k}-\delta^{0,1}\Phi_{*,k+1}\right)\cdot e^{\Phi/\hbar}\\
 =&-\frac{1}{\hbar}e^{\Phi/\hbar}\star\left(\delta^{0,1}(\Phi_{*,2})\right)+\frac{\sqrt{-1}}{2 \hbar}\omega^{i\bar{j}}\frac{\partial\Phi}{\partial y^i}\cdot\frac{\partial}{\partial\bar{y}^j}\left(\delta^{0,1}(\Phi_{*,2})\right)\cdot e^{\Phi/\hbar}\\
 &\hspace{6mm}+\frac{1}{\hbar}\sum_{k\geq 2}\left(\nabla^{0,1}\Phi_{*,k}-\delta^{0,1}\Phi_{*,k+1}\right)\cdot e^{\Phi/\hbar}\\
 =&-\frac{1}{\hbar}e^{\Phi/\hbar}\star\left(\delta^{0,1}(\Phi_{*,2})\right)\\
 &\hspace{6mm}+\frac{1}{\hbar}\sum_{k\geq 2}\left(\nabla^{0,1}\Phi_{*,k}-\delta^{0,1}\Phi_{*,k+1}+\frac{\sqrt{-1}}{2}\omega^{i\bar{j}}\frac{\partial\Phi_{*,k}}{\partial y^i}\cdot\frac{\partial}{\partial\bar{y}^j}\left(\delta^{0,1}(\Phi_{*,2})\right)\cdot e^{\Phi/\hbar}\right)\cdot e^{\Phi/\hbar}\\
 =&-\frac{1}{\hbar}e^{\Phi/\hbar}\star\left(\delta^{0,1}(\Phi_{*,2})\right).
 \end{align*}
 In the last line, we have used the following identity:
 $$
 \delta^{0,1}(\Phi_{m,n+1})=\nabla^{0,1}(\Phi_{m,n})+\frac{\sqrt{-1}}{2}\omega^{i\bar{j}}\frac{\partial\Phi}{\partial y^i}\cdot\frac{\partial}{\partial\bar{y}^j}\left(\delta^{0,1}(\Phi_{*,2})\right).
 $$
 Since $e^{\Phi/\hbar}\star(e^{\Phi/\hbar})^{-1}=1$, it is easy to show that:
 $$
 (\nabla^{0,1}-\delta^{0,1})(e^{\Phi/\hbar})^{-1}=\frac{1}{\hbar}\delta^{0,1}(\Phi_{*,2})\star(e^{\Phi/\hbar})^{-1}.
 $$
Using the fact that $\nabla^{0,1}-\delta^{0,1}$ is a derivation with respect to $\star$, we have 
\begin{align*}
&(\nabla^{0,1}-\delta^{0,1})\left(e^{\Phi/\hbar}\star O\star (e^{\Phi/\hbar})^{-1}\right)\\
=&e^{\Phi/\hbar}\star(-\frac{1}{\hbar}\delta^{0,1}(\Phi_{*,2}))\star O\star (e^{\Phi/\hbar})^{-1}+e^{\Phi/\hbar}\star(\nabla^{0,1}O-\delta^{0,1}O)\star (e^{\Phi/\hbar})^{-1}\\
&\qquad + e^{\Phi/\hbar}\star O\star\frac{1}{\hbar}(\delta^{0,1}(\Phi_{*,2})) \star(e^{\Phi/\hbar})^{-1}\\
=&e^{\Phi/\hbar}\star\left(\nabla^{0,1}O-\delta^{0,1}O+\frac{1}{\hbar}[-\delta^{0,1}(\Phi_{*,2}), O]_{\star}\right) \star(e^{\Phi/\hbar})^{-1}\\
=&e^{\Phi/\hbar}\star D_{F,\alpha}^{0,1}(O)\star(e^{\Phi/\hbar})^{-1}.
\end{align*}
In the last line, we have used equation \eqref{equation: relation-Phi-I} to obtain:
$$
-\delta^{0,1}(\Phi_{n,2})=-2\sqrt{-1}\cdot\delta^{0,1}(\Phi_\omega)_{n,2}=-2\sqrt{-1}\frac{\sqrt{-1}}{2}I_n=I_n.
$$
\end{proof}

\begin{prop}
 Let $O$ be any section of the Weyl bundle $\mathcal{W}_{X,\mathbb{C}}$, and let $O_q$ (q for quantization) be the unique solution of the following equation:
$$
O\cdot e^{\Phi/\hbar}=e^{\Phi/\hbar}\star O_q.
$$
Here $\Phi$ is the same as the previous part. Then there is the following identity describing an explicit relation between the classical and quantum (Fedosov) connections:
\begin{equation}\label{equation: relation-between-classical-quantum-connections}
D_C(A)\cdot e^{\Phi/\hbar}=e^{\Phi/\hbar}\star D_{F,\alpha}(A_q).
\end{equation}
\end{prop}
\begin{proof}
 Let $A$ and $B$ be sections of $\mathcal{W}_X$ and $\overline{\mathcal{W}}_X$ respectively, then so are the $D_C(A)$ and $D_C(B)$ as the classical connection $D_C$ does not change the type in $\mathcal{W}_{X,\mathbb{C}}$. For $A$, there is $A_q=A$ by type reason, and there is
 $$
 D_C(A)\cdot  e^{\Phi/\hbar}=e^{\Phi/\hbar}\star D_C(A_q)=e^{\Phi/\hbar}\star D_{F,\alpha}(A_q).
 $$
 The last equality follows from the fact that the Fedosov connection $D_{F,\alpha}$ equals $D_C$ when restricted to $\mathcal{W}_X$. For $B$, there is
 $$
 B\star e^{\Phi/\hbar}=B\cdot e^{\Phi/\hbar}=e^{\Phi/\hbar}\star B_q.
 $$
 By Lemma \ref{lemma: Fedosov-connection-jet-connection-conjugate}, there is
 \begin{align*}
  D_C^{0,1}(B)=&(\nabla^{0,1}-\delta^{0,1})(B)=(\nabla^{0,1}-\delta^{0,1})\left(e^{\Phi/\hbar}\star B_q\star(e^{\Phi/\hbar})^{-1}\right)=e^{\Phi/\hbar}\star D_{F,\alpha}^{0,1}(B_q)\star(e^{\Phi/\hbar})^{-1}.
 \end{align*}
 In a similar fashion, we can show that $D_C^{1,0}(B)=e^{\Phi/\hbar}\star D_{F,\alpha}^{1,0}(B_q)\star(e^{\Phi/\hbar})^{-1}$. A general monomial of $\mathcal{W}_{X,\mathbb{C}}$ must be a sum of the forms $A\cdot B$. We first have the following:
 \begin{align*}
   (A\cdot B)\cdot e^{\Phi/\hbar}=& A\cdot \left(B\cdot e^{\Phi/\hbar}\right)
  =A\cdot \left(e^{\Phi/\hbar}\star B_q\right)
  = e^{\Phi/\hbar}\star \left(B_q\star A\right),
 \end{align*}
 which implies that $(A\cdot B)_q=B_q\star A$. And there is 
 \begin{align*}
  D_C(A\cdot B)\cdot e^{\Phi/\hbar}
  & = (D_C(A)\cdot B+A\cdot D_C(B))\cdot e^{\Phi/\hbar}\\
  & = D_C(A)\cdot B\cdot e^{\Phi/\hbar}+A\cdot D_C(B)\cdot e^{\Phi/\hbar}\\
  & = D_C(A)\cdot(e^{\Phi/\hbar}\star B_q)+A\cdot (e^{\Phi/\hbar}\star D_{F,\alpha}(B_q))\\
  & = (e^{\Phi/\hbar}\star B_q)\star D_C(A)+A\cdot (e^{\Phi/\hbar}\star D_{F,\alpha}(B_q))\\
  & = (e^{\Phi/\hbar}\star B_q)\star D_{F,\alpha}(A)+(e^{\Phi/\hbar}\star D_{F,\alpha}(B_q))\star A\\
  & = e^{\Phi/\hbar}\star (B_q\star D_{F,\alpha}(A)+D_{F,\alpha}(B_q)\star A)\\
  & = e^{\Phi/\hbar}\star D_{F,\alpha}(B_q\star A)\\
  & = e^{\Phi/\hbar}\star D_{F,\alpha}(A\cdot B)_q.
 \end{align*}
\end{proof}
This proposition reduces the proof of Theorem \ref{thm:main-thm-formal-Toeplitz-Fedosov-equation} to showing that $\sigma(O_f)=f$. This follows the definition of $O_f$ and the fact that the section of $\Phi$ does not contain any non-trivial purely holomorphic or anti-holomorphic components. Thus we complete the proof of Theorem \ref{thm:main-thm-formal-Toeplitz-Fedosov-equation}.

\section{Proof of Theorem \ref{theorem: flat-section-Bargmann-Fock}}\label{appendix: flat-section-Bargmann-Fock}

Using Lemma \ref{lemma: compatibility-Fedosov-connection} and the fact that $D_{F,\alpha}|_{\W_X}=D_K$, we have
\begin{align*}
 D_{B,\alpha}(A\cdot e^{\beta/\hbar}\otimes e_{x_0})
 & = D_{B,\alpha}(A\circledast e^{\beta/\hbar}\otimes e_{x_0})\\
 & = D_{F,\alpha}(A)\circledast (e^{\beta/\hbar}\otimes e_{x_0})+A\circledast D_{B,\alpha}(e^{\beta/\hbar}\otimes e_{x_0})\\
 & = D_K(A)\cdot (e^{\beta/\hbar}\otimes e_{x_0})+A\circledast D_{B,\alpha}(e^{\beta/\hbar}\otimes e_{x_0}).
\end{align*}
Hence, to prove the theorem, we only need to show that $D_{B,\alpha}(e^{\beta/\hbar}\otimes e_{x_0})=0$.
We first recall that $\alpha=-\hbar\cdot R_{i\bar{j}k}^kdz^i\wedge d\bar{z}^j$.
\begin{lem}\label{lemma: trace-I-Ricci-form}
We have
$
(J_\alpha)_n=-(n+1)\hbar\cdot R_{ii_1\cdots i_{n},\bar{l}}^id\bar{z}^l\otimes y^{i_1}\cdots y^{i_n}
$
\end{lem}
\begin{proof}
The proof is by induction on $n$.
For $n=1$, we have
$$
(J_\alpha)_1=(\delta^{1,0})^{-1}\left(-\hbar\cdot R_{i\bar{j}k}^kdz^i\wedge d\bar{z}^j\right)=-2\hbar\cdot\left(\frac{1}{2} R_{i\bar{j}k}^kd\bar{z}^j\otimes y^i\right).
$$
Then by the induction hypothesis for $n-1$, we have
\begin{align*}
 \nabla^{1,0}(J_\alpha)_{n-1}=&\nabla^{1,0}\left(-n\hbar\cdot R_{ii_1\cdots i_{n-1},\bar{l}}^id\bar{z}^l\otimes y^{i_1}\cdots y^{i_{n-1}}\right).
\end{align*}
On the other hand, 
\begin{align*}
 \nabla^{1,0}\left(n\hbar\cdot R_{i_1\cdots i_{n},\bar{l}}^jd\bar{z}^l\otimes y^{i_1}\cdots y^{i_{n}}\otimes\partial_{y^j}\right)
 =(n+1)\cdot n\hbar\cdot R_{i_1\cdots i_{n+1},\bar{l}}^jdz^{i_{n+1}}\wedge d\bar{z}^l\otimes y^{i_1}\cdots y^{i_n}\otimes\partial_{y^j}.
\end{align*}
Since $\nabla^{1,0}$ is compatible between the contraction between $TX$ and $T^*X$, the above computation shows that
$$
(J_\alpha)_n=(\delta^{1,0})^{-1}(\nabla^{1,0}(J_\alpha)_{n-1})=-(n+1)\hbar\cdot R_{i_1\cdots i_{n+1},\bar{l}}^{i_1} d\bar{z}^l\otimes y^{i_1}\cdots y^{i_n}y^{i_{n+1}}.
$$
\end{proof}
\begin{lem}
 The section $\beta$ satisfies 
 $
 D_K(\beta)=2\sqrt{-1}\omega_{i\bar{j}}d\bar{z}^j\otimes y^i-\partial\rho.
 $
\end{lem}
\begin{proof}
The function $\rho$ satisfies the condition that $\partial\bar{\partial}(\rho)=-2\sqrt{-1}\omega$. Recall that $\beta=\sum_{k\geq 1}(\tilde{\nabla}^{1,0})^k(\rho)$. A straightforward computation shows that 
\begin{align*}
D_K(\beta)=&(-\delta^{1,0}+\bar{\partial})(\tilde{\nabla}^{1,0}\rho)=-\partial\rho+\bar{\partial}\circ(\delta^{1,0})^{-1}(\nabla^{1,0}\rho)\\
=&-\partial\rho+(\delta^{1,0})^{-1}(\bar{\partial}\partial\rho)=2\sqrt{-1}\omega_{i\bar{j}}d\bar{z}^j\otimes y^i-\partial\rho.
\end{align*}
\end{proof}
We also have the following:
\begin{align*}
 &\frac{1}{\hbar}I_n\circledast (e^{\beta/\hbar}\otimes e_{x_0})\\
 =&-2\sqrt{-1}\cdot R_{i_1\cdots i_n,\bar{l}}^j\omega_{j\bar{k}}d\bar{z}^l\otimes(y^{i_1}\cdots y^{i_n}\bar{y}^k)\circledast(e^{\beta/\hbar}\otimes e_{x_0})\\
 =&-2\sqrt{-1}\cdot R_{i_1\cdots i_n,\bar{l}}^j\omega_{j\bar{k}}d\bar{z}^l\otimes(\frac{\omega^{i\bar{k}}}{2\sqrt{-1}}\frac{\partial}{\partial y^i})(y^{i_1}\cdots y^{i_n}e^{\beta/\hbar}\otimes e_{x_0})\\
 =&R_{i_1\cdots i_n,\bar{l}}^id\bar{z}^l\otimes y^1\cdots y^n\frac{\partial(\beta/\hbar))}{\partial y^i}\cdot(e^{\beta/\hbar}\otimes e_{x_0})+n\cdot R_{ii_1\cdots i_{n-1},\bar{l}}^id\bar{z}^l\otimes y^{i_1}\cdots y^{i_{n-1}}\cdot(e^{\beta/\hbar}\otimes e_{x_0})\\
 =&\tilde{R}_n^*(\beta/\hbar)+n\cdot R_{ii_1\cdots i_{n-1},\bar{l}}^id\bar{z}^l\otimes y^{i_1}\cdots y^{i_{n-1}}\cdot(e^{\beta/\hbar}\otimes e_{x_0}).
\end{align*}
Summarizing the above computations, we have
\begin{align*}
 &D_{B,\alpha}(e^{\beta/\hbar}\otimes e_{x_0})\\
 =&\left(\nabla+\frac{1}{\hbar}\gamma_\alpha\circledast\right)(e^{\beta/\hbar}\otimes e_{x_0})+e^{\beta/\hbar}\otimes \nabla_{L_{\omega/\hbar}}e_{x_0}\\
 =&\left(\nabla(\beta/\hbar)+\frac{2\sqrt{-1}}{\hbar}\omega_{i\bar{j}}(dz^i\otimes\bar{y}^i-d\bar{z}^j\otimes y^i)\circledast+\frac{1}{\hbar}(I+J_\alpha)\circledast\right)(e^{\beta/\hbar}\otimes e_{x_0})+e^{\beta/\hbar}\otimes \nabla_{L_{\omega/\hbar}}e_{x_0}\\
 =&\left(\nabla(\beta/\hbar)-\frac{2\sqrt{-1}}{\hbar}\omega_{i\bar{j}}d\bar{z}^j\otimes y^i+\frac{1}{\hbar}\partial\rho\right)(e^{\beta/\hbar}\otimes e_{x_0})\\
 &\qquad+\frac{1}{\hbar}(2\sqrt{-1}\omega_{i\bar{j}}dz^i\otimes\bar{y}^j+I+J_\alpha)\circledast(e^{\beta/\hbar}\otimes e_{x_0})\\
 =&\left(\nabla(\beta/\hbar)+\sum_{n\geq 2}\tilde{R}_n^*(\beta/\hbar)-\frac{2\sqrt{-1}}{\hbar}\omega_{i\bar{j}}d\bar{z}^j\otimes y^i+\frac{1}{\hbar}\partial\rho+2\sqrt{-1}\omega_{i\bar{j}}dz^i\frac{\omega^{k\bar{j}}}{2\sqrt{-1}}\frac{\partial(\beta/\hbar)}{\partial y^k}\right)(e^{\beta/\hbar}\otimes e_{x_0})\\
  =&\left(\nabla(\beta/\hbar)+\sum_{n\geq 2}\tilde{R}_n^*(\beta/\hbar)-\frac{2\sqrt{-1}}{\hbar}\omega_{i\bar{j}}d\bar{z}^j\otimes y^i+\frac{1}{\hbar}\partial\rho-\delta^{1,0}(\beta/\hbar)\right)(e^{\beta/\hbar}\otimes e_{x_0})\\
  =&\frac{1}{\hbar}\left(D_K(\beta)-2\sqrt{-1}\omega_{i\bar{j}}d\bar{z}^j\otimes y^i+\partial\rho\right)(e^{\beta/\hbar}\otimes e_{x_0})\\
  =&0.
\end{align*}
This completes the proof of Theorem \ref{theorem: flat-section-Bargmann-Fock}.


\begin{bibdiv}
\begin{biblist}

\bib{BGKP}{article}{
    AUTHOR = {Baranovsky, V.},
    author = {Ginzburg, V.},
    author = {Kaledin, D.},
    author = {Pecharich, J.},
     TITLE = {Quantization of line bundles on lagrangian subvarieties},
   JOURNAL = {Selecta Math. (N.S.)},
    VOLUME = {22},
      YEAR = {2016},
    NUMBER = {1},
     PAGES = {1--25},
}

\bib{Bordemann}{article}{
    AUTHOR = {Bordemann, M.},
    author = {Waldmann, S.},
     TITLE = {A {F}edosov star product of the {W}ick type for {K}\"{a}hler
              manifolds},
   JOURNAL = {Lett. Math. Phys.},
    VOLUME = {41},
      YEAR = {1997},
    NUMBER = {3},
     PAGES = {243--253},
}

\bib{Bordemann-Waldmann}{article}{
    AUTHOR = {Bordemann, M.},
    author = {Waldmann, S.},
     TITLE = {Formal {GNS} construction and states in deformation
              quantization},
   JOURNAL = {Comm. Math. Phys.},
    VOLUME = {195},
      YEAR = {1998},
    NUMBER = {3},
     PAGES = {549--583},
}

\bib{Bordemann-Meinrenken}{article}{
    AUTHOR = {Bordemann, M.},
    author = {Meinrenken, E.},
    author = {Schlichenmaier, M.},
     TITLE = {Toeplitz quantization of {K}\"{a}hler manifolds and {${\rm
              gl}(N)$}, {$N\to\infty$} limits},
   JOURNAL = {Comm. Math. Phys.},
    VOLUME = {165},
      YEAR = {1994},
    NUMBER = {2},
     PAGES = {281--296},
}


\bib{Chan-Leung-Li}{article}{
   author={Chan, K.},
   author={Leung, N.},
   author={Li, Q.},
   title={A geometric construction of representations of the Berezin-Toeplitz quantization},
   eprint={arXiv:2004.00523 [math-QA]},
}

\bib{CLL}{article}{
   author={Chan, K.},
   author={Leung, N.},
   author={Li, Q.},
   title={Kapranov's $L_\infty$ structures, Fedosov's star products, and one-loop exact BV quantizations on K\"ahler manifolds},
   eprint={ arXiv:2008.07057 [math-QA]},
}


\bib{Kevin-book}{book}{
   author={Costello, K.},
   title={Renormalization and effective field theory},
   series={Mathematical Surveys and Monographs},
   volume={170},
   publisher={American Mathematical Society},
   place={Providence, RI},
   date={2011},
   pages={viii+251},
   isbn={978-0-8218-5288-0},
}

\bib{DLS}{article}{
     author= {Dolgushev, V. A. },
     author= {Lyakhovich, S. L.},
     author= {Sharapov, A. A.},
     TITLE = {Wick type deformation quantization of {F}edosov manifolds},
   JOURNAL = {Nuclear Phys. B},
    VOLUME = {606},
      YEAR = {2001},
    NUMBER = {3},
     PAGES = {647--672},
}


\bib{Donaldson}{incollection}{
	AUTHOR = {Donaldson, S. K.},
	TITLE = {Planck's constant in complex and almost-complex geometry},
	BOOKTITLE = {X{III}th {I}nternational {C}ongress on {M}athematical
		{P}hysics ({L}ondon, 2000)},
	PAGES = {63--72},
	PUBLISHER = {Int. Press, Boston, MA},
	YEAR = {2001},
}


\bib{Fed}{article}{
    AUTHOR = {Fedosov, B. V.},
     TITLE = {A simple geometrical construction of deformation quantization},
   JOURNAL = {J. Differential Geom.},
    VOLUME = {40},
      YEAR = {1994},
    NUMBER = {2},
     PAGES = {213--238}
}


\bib{Fedbook}{book}{
    AUTHOR = {Fedosov, B. V.},
     TITLE = {Deformation quantization and index theory},
    SERIES = {Mathematical Topics},
    VOLUME = {9},
 PUBLISHER = {Akademie Verlag, Berlin},
      YEAR = {1996},
     PAGES = {325},
}



\bib{Kapranov}{article}{
    AUTHOR = {Kapranov, M.},
     TITLE = {Rozansky-{W}itten invariants via {A}tiyah classes},
   JOURNAL = {Compositio Math.},
    VOLUME = {115},
      YEAR = {1999},
    NUMBER = {1},
     PAGES = {71--113},
}



\bib{Karabegov96}{article}{
    AUTHOR = {Karabegov, A.V.},
     TITLE = {Deformation quantizations with separation of variables on a
              {K}\"{a}hler manifold},
   JOURNAL = {Comm. Math. Phys.},
    VOLUME = {180},
      YEAR = {1996},
    NUMBER = {3},
     PAGES = {745--755},
}

\bib{Karabegov00}{incollection}{
    AUTHOR = {Karabegov, A.V.},
     TITLE = {On {F}edosov's approach to deformation quantization with
              separation of variables},
 BOOKTITLE = {Conf\'{e}rence {M}osh\'{e} {F}lato 1999, {V}ol. {II} ({D}ijon)},
    SERIES = {Math. Phys. Stud.},
    VOLUME = {22},
     PAGES = {167--176},
 PUBLISHER = {Kluwer Acad. Publ., Dordrecht},
      YEAR = {2000},
}

\bib{Karabegov07}{article}{
    AUTHOR = {Karabegov, A.V.},
     TITLE = {A formal model of {B}erezin-{T}oeplitz quantization},
   JOURNAL = {Comm. Math. Phys.},
    VOLUME = {274},
      YEAR = {2007},
    NUMBER = {3},
     PAGES = {659--689},
}

\bib{Karabegov}{article}{
    AUTHOR = {Karabegov, A.V.},
    author = {Schlichenmaier, M.},
     TITLE = {Identification of {B}erezin-{T}oeplitz deformation
              quantization},
   JOURNAL = {J. Reine Angew. Math.},
    VOLUME = {540},
      YEAR = {2001},
     PAGES = {49--76},
}



\bib{Kontsevich2}{article}{
    AUTHOR = {Kontsevich, M.},
     TITLE = {Rozansky-{W}itten invariants via formal geometry},
   JOURNAL = {Compositio Math.},
    VOLUME = {115},
      YEAR = {1999},
    NUMBER = {1},
     PAGES = {115--127},
}



\bib{Ma-Ma-1}{article}{
    AUTHOR = {Ma, X.},
    author = {Marinescu, G.},
     TITLE = {Toeplitz operators on symplectic manifolds},
   JOURNAL = {J. Geom. Anal.},
    VOLUME = {18},
      YEAR = {2008},
    NUMBER = {2},
     PAGES = {565--611},
}

\bib{Ma-Ma}{article}{
    AUTHOR = {Ma, X.},
    author = {Marinescu, G.},
     TITLE = {Berezin-{T}oeplitz quantization on {K}\"{a}hler manifolds},
   JOURNAL = {J. Reine Angew. Math.},
    VOLUME = {662},
      YEAR = {2012},
     PAGES = {1--56},

}

\bib{Melrose}{article}{
    AUTHOR = {Melrose, R.},
     TITLE = {Star products and local line bundles},
   JOURNAL = {Ann. Inst. Fourier (Grenoble)},
    VOLUME = {54},
      YEAR = {2004},
    NUMBER = {5},
     PAGES = {1581--1600, xvi, xxii},
}

\bib{Nest-Tsygan}{article}{
    AUTHOR = {Nest, R.},
    author = {Tsygan, B.},
     TITLE = {Remarks on modules over deformation quantization algebras},
   JOURNAL = {Mosc. Math. J.},
    VOLUME = {4},
      YEAR = {2004},
    NUMBER = {4},
     PAGES = {911--940, 982},
}

\bib{Neumaier}{article}{
    AUTHOR = {Neumaier, N.},
     TITLE = {Universality of {F}edosov's construction for star products of
              {W}ick type on pseudo-{K}\"{a}hler manifolds},
   JOURNAL = {Rep. Math. Phys.},
    VOLUME = {52},
      YEAR = {2003},
    NUMBER = {1},
     PAGES = {43--80},
}



\bib{Tian}{article}{
    AUTHOR = {Tian, G.},
     TITLE = {On a set of polarized {K}\"{a}hler metrics on algebraic manifolds},
   JOURNAL = {J. Differential Geom.},
    VOLUME = {32},
      YEAR = {1990},
    NUMBER = {1},
     PAGES = {99--130},
}

\bib{Tsygan}{article}{
    AUTHOR = {Tsygan, B.},
     TITLE = {Oscillatory modules},
   JOURNAL = {Lett. Math. Phys.},
    VOLUME = {88},
      YEAR = {2009},
    NUMBER = {1-3},
     PAGES = {343--369},
}

\bib{Zelditch}{article}{
	AUTHOR = {Zelditch, S.},
	TITLE = {Szego kernels and a theorem of {T}ian},
	JOURNAL = {Internat. Math. Res. Notices},
	YEAR = {1998},
	NUMBER = {6},
	PAGES = {317--331},
}



\end{biblist}
\end{bibdiv}
\end{document}